\documentclass{conm-p-l}

\usepackage{amssymb}
\usepackage{graphicx}
\usepackage{enumerate}
\usepackage[pdftex,all]{xy}
\usepackage[english]{babel}
\usepackage[latin1]{inputenc}
\usepackage{pgf,tikz}
\usetikzlibrary[patterns]

\usepackage{tikz-cd}
\usetikzlibrary{babel}
\usepackage{paralist,setspace,a4,mathrsfs,dsfont,MnSymbol}
\usepackage{float}
\usepackage{varioref}
\usepackage[bookmarksnumbered=true]{hyperref}
\usepackage[capitalize]{cleveref}

\usepackage[resetlabels]{multibib}
\newcites{App}{References}
\newcites{Main}{References}

\newtheorem{thm}{Theorem}[section]
\newtheorem{lem}[thm]{Lemma}
\newtheorem{prop}[thm]{Proposition}
\newtheorem{cor}[thm]{Corollary}

\theoremstyle{definition}
\newtheorem{defn}[thm]{Definition}

\newtheorem{exmp}[thm]{Example}
\newtheorem{nota}[thm]{Notation}

\theoremstyle{remark}
\newtheorem{rem}[thm]{Remark}

\numberwithin{equation}{section}

\renewcommand{\P}{\mathbf{P}}

\newcommand{\Xbar}{\overline{X}}
\newcommand{\Ybar}{\overline{Y}}
\newcommand{\Cbar}{\overline{C}}
\newcommand{\Dbar}{\overline{D}}
\newcommand{\hirz}[1]{\Sigma_{#1}}
\newcommand{\spec}{\textrm{Spec}}
\newcommand{\Bl}[2]{\textrm{Bl}_{#1}(#2)}
\newcommand{\cP}{\mathcal{P}}
\newcommand{\cQ}{\mathcal{Q}}
\newcommand{\Jac}{\textrm{Jac}}
\newcommand{\Spec}{\textrm{Spec}}

\renewcommand{\L}{\mathcal{L}}
\renewcommand{\O}{\mathcal{O}}
\newcommand{\I}{\mathcal{I}}

\newcommand{\F}{\mathbf{F}}
\newcommand{\Fq}{\F_q}
\newcommand{\Fqbar}{\overline{\F}_q}

\renewcommand{\H}{\mathrm{H}}
\newcommand{\Hom}{\mathrm{Hom}}

\newcommand{\AGCode}[3]{C(#1,#2,#3)}

\newcommand{\wt}[1]{\textrm{w}_H\left(#1\right)}
\newcommand{\suppt}[1]{\ensuremath{\textbf{Supp}(#1)}}

\newcommand{\sesh}[3]{\varepsilon(#1,#2,#3)} 
 
\newcommand{\map}[4]{\left\{
\begin{array}{ccc}
#1 & \longrightarrow & #2 \\
#3 & \longmapsto     & #4
        \end{array}
      \right.} 
\newcommand{\CH}{\textrm{CH}}

\renewcommand{\geq}{\geqslant}
\renewcommand{\leq}{\leqslant}

\def\ds{{\displaystyle}}
\def\CF{{\mathcal F}}

\def\CG{{\mathcal G}}

\def\ZZ{{\mathbb Z}}
\def\hX{{{X}_{x}^h}}
\def\hU{{U}_{x}}
\def\Q{{\mathbf Q}}
\def\commutatif{\ar@{}[rd]|{\circlearrowleft}}

\renewcommand{\O}{\mathcal{O}}

\newcommand{\liso}{\mathrel{\hbox{$\longrightarrow$} \kern-2.4ex\lower-1ex\hbox{$\scriptstyle\sim$}\kern1.7ex}}

 \makeatletter
\newcommand{\addresseshere}{
  \enddoc@text\let\enddoc@text\relax
}
\makeatother

\begin{document}

\title{Toward good families of codes from towers of
  surfaces}

\author{Alain Couvreur}
\address{INRIA \& Laboratoire LIX,
  CNRS UMR 7161,\'Ecole Polytechnique, 91120 {\sc Palaiseau Cedex}}
\curraddr{}
\email{alain.couvreur@inria.fr}
\thanks{}

\author{Philippe Lebacque}
\address{Laboratoire de Mathématiques de Besan\c{c}on, UFR Sciences et
  techniques, 16 route de Gray, 25030 Besançon, France}
\curraddr{}
\email{philippe.lebacque@univ-fcomte.fr}
\thanks{}

\author{Marc Perret} \address{Marc Perret, Institut de Math\'ematiques
  de Toulouse, UMR 5219, Universit\' e de Toulouse, CNRS, UT2J,
  F-31058 Toulouse, France} \email{perret@univ-tlse2.fr}
\thanks{Funded by ANR grant ANR-15-CE39-0013 ``Manta'', the ANR grant
  ANR-17-CE40-0012 ``Flair'' and the BFC grant GA CROCOCO}


\date{\today}
\dedicatory{To our late teacher, colleague and friend\\ Gilles Lachaud}

\begin{abstract}
  We introduce in this article a new method to estimate the minimum
  distance of codes from algebraic surfaces. This lower bound is
  generic, i.e. can be applied to any surface, and turns out to be
  ``liftable'' under finite morphisms, paving the way toward the
  construction of good codes from towers of surfaces. In the same
  direction, we establish a criterion for a surface with a fixed
  finite set of closed points $\cP$ to have an infinite tower of
  $\ell$--étale covers in which $\cP$ splits totally.  We conclude by
  stating several open problems. In particular, we relate the
  existence of asymptotically good codes from general type surfaces
  with a very ample canonical class to the behaviour of their number
  of rational points with respect to their $K^2$ and coherent Euler
  characteristic.
\end{abstract}

\maketitle

\setcounter{tocdepth}{1}
\tableofcontents

\section*{Introduction}
In the early 80's, V.~D. Goppa proposed a construction of codes from
algebraic curves \citeMain{goppa}.  A pleasant feature of these codes from
curves is that they benefit from an elementary but rather sharp lower
bound for the minimum distance, the so--called {\em Goppa designed
  distance}. In addition, an easy lower bound for the dimension can be
derived from Riemann-Roch theorem.  The very simple nature of these
lower bounds for the parameters permits to formulate a concise
criterion for a sequence of curves to provide asymptotically good
codes~: roughly speaking, {\em the curves should have a large number
  of rational points compared to their genus}. Ihara
\citeMain{Ihara} and independently  Tsfasman,
Vl\u{a}du\c{t} and Zink \citeMain{TVZ} proved that this criterion was satisfied by some modular
and Shimura curves. This led to an impressive breaktrough in coding
theory~: the unexpected existence of sequences of codes beating the
asymptotic Gilbert-Varshamov bound.

As suggested by Manin and Vl\u{a}du\c{t} \citeMain{manin}, Goppa's
original construction extends to higher dimensional
varieties. However, in comparison to the rich literature involving
codes from curves, codes from higher dimensional varieties have been
subject to very few developments. One of the reasons for this gap is
that for varieties of dimension two and higher, divisors and points
are objects of different nature, and this difference makes the estimate of
the minimum distance much more difficult.  For this reason, most of
the works on codes from surfaces in the literature concern specific
subfamilies of surfaces, and estimate the parameters by ad hoc
techniques relying on specific arithmetic and geometric properties of
the involved surfaces. For instance, codes from quadric and Hermitian
surfaces are considered in \citeMain{fred, fred2, CouDuu}, toric surface
codes are studied in \citeMain{hansen} and codes from Hirzebruch surfaces
in \citeMain{nardi19}, codes from rational surfaces obtained by specific
blowups of the projective planes are analysed in \citeMain{Cou11, davis,
  ballico11, ballico13} and codes from abelian surfaces are
investigated in \citeMain{haloui17, ABHP19}. Toward an analog of {\em
  having a large number of points compared to their genus}, Voloch and
Zarzar suggested to seek surfaces with a low Picard number
\citeMain{zarzar, agctvoloch}, leading to further discoveries of good
codes \citeMain{ls18, blache19}. A significant part of the previously
described code constuctions appear in the excellent survey
\citeMain{Little_Survey}.

On the other hand, a few references propose ``generic'' lower bounds
for the minimum distance, i.e. bounds that could be applied to codes
from any surface. The first one is proposed by Aubry in
\citeMain{Aubry}, resting on a counting argument on the maximum
number of points of a curve in a very ample linear system. Next,
S.~H. Hansen proposed two other approaches to bound the minimum
distance from below \citeMain{soH}. The first one  involves an
auxiliary set of curves, while the second one  involves the Seshardi
constant at the evaluation points. Note that the latter approach is very close
to that of Bouganis \citeMain{bouganis}, while Bouganis' code construction
is not {\em stricto sensu} an algebraic geometry code from a surface.
Unfortunately, it should be emphasized that, even with the knowledge
of these bounds, no result is known on the possible asymptotic
performance of codes from algebraic surfaces, and the principle of
getting good sequences of codes from sequences of curves {\em whose
  number of rational points grows quickly compared to their genus} has
no known counterpart in the surfaces' setting.

It is worth noting that besides the highly difficult objective of
getting new families of codes with excellent asymptotic parameters,
the geometry of algebraic surfaces could provide codes with many other
interesting features. Indeed, in the last decade, new questions
brought by distributed storage problems arose and motivated a focus on
codes with {\em good local properties}. Concerning the use of
algebraic geometry codes for constructing locally recoverable codes,
one can cite for instance (and the list is very far from being
exhaustive) \citeMain{BT14, BTZ17, BHHMVA17, HMM18, LMX19, MTT19,
  SVV19}.  It should be pointed out that algebraic geometry codes from
surfaces are of deep interest from this point of view. Indeed, given a
code from a surface $X$, to access to some digit of the code, which
corresponds to some point of the surface, one can consider a curve
contained in the surface and containing this point and restrict to the
code on the curve. This point of view, which is in particular investigated in
\citeMain{SVV19}, provides another motivation to
study codes from surfaces or higher dimensional varieties.
Note finally that these interesting local properties of
codes on surfaces have been remarked more than 15 years ago by
Bouganis \citeMain{bouganis}.

\subsection*{Our contribution} This article expects to be a first step
toward the construction of asymptotically good families codes from
sequences of algebraic surfaces. For this sake, we provide several
tools and results that would be helpful to progress in this direction.

In \S~\ref{sec:codes_on_surfaces}, we propose a new ``generic'' lower
bound for the minimum distance of an algebraic geometry code from a
surface. This bound involves a linear system of divisor with a
property we called {\em $\cP$--interpolating} (see
Definition~\ref{def:P-interpolating}).  The existence and the
construction of $\cP$--interpolating divisors could appear as a
limitation to handle this result. However, we also prove that for a
surface $X$ over the finite field $\F_q$, the linear system associated
to $\O_X(q+1)$ is always $\cP$--covering.  Further we compare this new
bound with previous ones in the literature.  Finally, we discuss the
behaviour of these generic bounds, including ours, under finite
morphisms of projective surfaces and show that our bound such as those
involving Seshadri constants are easy to ``lift'' under such a finite
map, and hence could be used to study codes from towers of surfaces.

In \S~\ref{sec:philippe}, we propose a class field--like criterion for
a surface with a fixed set of rational points $\cP$ to have an
infinite tower of $\ell$--étale covers in which $\cP$ splits totally.
Next, in \S~\ref{sec:classification} we sieve the Kodaiara
classification of algebraic surfaces to determine which types of
surfaces may have infinite totally split towers of étale covers.
Despite we acknowledge that we have not been able to apply the
criterion established in \S~\ref{sec:philippe} in order to provide new
families of codes, we however show in \S~\ref{sec:example} that this
criterion can be applied on a product of two hyperelliptic curves.  

We conclude the article by presenting some open problems in
\S~\ref{sec:open_prob}. In particular, we show that the existence of
asymptotically good families of codes can be deduced from that of
families of general type surfaces with very ample canonical class $K$
whose number $N$ of rational points goes to infinity, whose $K^2$ and
coherent Euler characteristic $\chi(\O_X)$ are proportional to $N$
and whose asymptotic ratio $\chi (\O_X)/K_X$ is the largest possible.

\subsection*{Note} The published version of the present article
includes an appendix by Alexander Schmidt \citeMain{schmidt3}.

\section{Codes from surfaces}\label{sec:codes_on_surfaces}
\subsection{Context and notation}
\subsubsection{Context}
In what follows, $X$ denotes a smooth projective geometrically
connected surface over $\F_q$ and $\cP$ a non empty set of rational
points on $X$ of cardinality $n$.  The surface $\Xbar$ is defined as
$\Xbar := X \times_{\spec(\F_q)} \spec (\Fqbar)$.  In addition, $G$ is
a divisor on $X$ whose support avoids the elements of $\cP$.
Our point is to study the parameters of the linear
code $\AGCode{X}{\cP}{G}$ introduced in \citeMain{goppa, manin} defined as the
image of the map:
\[
\textrm{ev}_\cP : \map{\H^{0}({X},{G})}{\F_q^n}{f}{(f(P))_{P
    \in \cP}},
\]
where elements of $\H^{0}({X}, {G})$ are regarded as rational fractions on $X$.

\begin{rem}
  Actually, the condition ``the support of $G$ avoids the elements of
  $\cP$'' can be removed. In such a case, one needs to choose at each
  point $P$ a generator $s_P$ of the stalk of the sheaf $\O(G)$ at
  $P$. Then, any global section $f$ can be locally written $f_P s_P$
  for some $f_P \in \O_{X,P}$, and one can evaluate $f_P$ at $P$.
\end{rem}

\subsubsection{Codes}
Recall that a {\em code} is a vector subspace $C$ of $\F_q^n$ for some finite
field $\F_q$ and some positive integer $n$, called the {\em length} of $C$.
The {\em dimension} of the code is the dimension of $C$ as a $\F_q$-vector space.
The {\em Hamming weight} $\wt{\mathbf x}$ of a vector
$\mathbf x \in \F_q^n$ is the number of nonzero entries of $\mathbf x$,
and that the {\em minimum distance} of a code $C \subseteq \F_q^n$ is the
minimum weight of a nonzero vector of $C$.
The objective of this first section is to investigate various manners
to estimate the parameters of codes from algebraic surfaces, namely the
dimension and the minimum distance.

\subsubsection{Equivalence of divisors and cycles}
In the sequel, we frequently deal with linear and numerical
equivalence of divisors. One is denoted by $\sim$ and the
second one by $\equiv$. We refer to \citeMain[Chapters II.6 \&
V.1]{Hartshorne} for the definitions. Recall that linear equivalence
entails numerical equivalence. Further we also deal with rational
equivalence of cycles which is a generalisation of linear equivalence
of divisors and hence will also be denoted by $\sim$. We refer to
\citeMain{Fulton_intersection} for the definition of
rational equivalence.

\subsection{The dimension of codes from surfaces}
To bound from below the dimension of a code $\AGCode{X}{\cP}{G}$, the
natural tool is Riemann-Roch theorem which asserts that
\[
h^0(X, G) - h^1(X,G) +h^2 (X,G) = \frac 1 2 G \cdot (G-K_X) + \chi (\O_X),
\]
and hence
\[
h^0(X, G) + h^2 (X,G) \geq \frac 1 2 G \cdot (G-K_X) + \chi (\O_X).
\]
The following lemma is useful in what follows.  This criterion can be
found in \citeMain{Hartshorne} and has been previously suggested for
applications to codes on surfaces by Bouganis in \citeMain{bouganis}.

\begin{lem}\label{lem:dimension}
  Let $H$ be an ample divisor on $X$. If $G \cdot H  > K_X \cdot H $ then
    \[
    h^0(X,G) \geq \frac 1 2 G \cdot (G - K_X) + \chi (\O_X).
    \]
\end{lem}

\begin{proof}
  From \citeMain[Lemma V.1.7]{Hartshorne}, if $G \cdot H  > K_X \cdot H $ then
  $h^2(X,G) = 0$ which yields the proof.
\end{proof}

\subsection{A new lower bound for the minimum distance}
In this subsection, we present a general manner to bound from below
the minimum distance of a code from a surface.  This bound can be
regarded as the counterpart for codes from surfaces of Goppa's
designed distance for codes from curves.  Indeed, in the case of
curves, an algebraic geometry code of length $n$ associated to a
divisor $G$ has minimum distance bounded from below by $n - \deg
G$. In the sequel, we prove that on a surface, a code of length $n$
associated to a divisor $G$ has minimum distance bounded from below by
$n - \Gamma \cdot G$ for some divisor $\Gamma$ with a special property. Here, such as in many
situations in algebraic geometry, the degree for divisors on a curve
is replaced by the intersection product for divisors on surfaces.  For
this sake, the introduction of an auxiliary divisor $\Gamma$ is
necessary. Our work to follow consists in providing de relevant
definition for $\Gamma$. Therefore, the statement of this new bound
requires first to recall some basic notions on linear systems of
divisors.

\subsubsection{Linear systems on a surface}\label{ss:lin_syst}
Recall that a {\em linear system} of divisors $\Gamma$ on $\Xbar$ is a
family of positive divisors which are all linearly equivalent, and
which is parametrised by a projective space. Given a divisor $D$ on
$\Xbar$, the {\em complete linear system} denoted as $|D|$ is the set
of all positive divisors on $\Xbar$ which are linearly equivalent to
$D$. This set is parametrised by $\P (\H^{0}({\Xbar},{D}))$ and a
general linear system is a subset of some complete linear system $|D|$
paramatrised by some proper linear subspace of
$\P (\H^{0}({\Xbar}, {D}))$.

\subsubsection{$\cP$--interpolating linear systems}

\begin{nota}
  Let $\Gamma$ be a linear system of curves on $\Xbar$ and let
  $\Ybar$ be a proper closed sub-scheme of $\Xbar$. Then, we denote
  by $\Gamma - \Ybar$ the maximal linear subsystem of $\Gamma$ of elements
  whose base locus contains $\Ybar$.
\end{nota}

\begin{defn}\label{def:P-interpolating}
  Given a surface $X$ and a set of rational points $\cP$, a linear system
  $\Gamma$ of divisors on $\Xbar$ is said to be {\em $\cP$--interpolating} if
  \begin{enumerate}[(i)]
  \item\label{item:non_empty} $\Gamma - \cP$ is non empty;
  \item\label{item:zero_dim} the base locus of $\Gamma - \cP$
    has dimension $0$.
  \end{enumerate}
\end{defn}

\begin{rem}\label{rem:base_locus}
  About condition~(\ref{item:zero_dim}), obviously, the base locus of
  $\Gamma - \cP$ contains $\cP$. In addition it may contain a finite
  number of other points. In particular, this base locus cannot be empty.
\end{rem}

\begin{rem}
  Note that the notion of $\cP$--intepolating system such as lower
  bound for the minimum distance we get in Theorem~\ref{thm:bound} to
  follow are of {\bf geometric} and not arithmetic nature: the
  $\cP$-interpolating linear system $\Gamma$ is \textbf{not} assumed
  to be defined over $\F_q$.
\end{rem}

The following statement provides an equivalent definition of
$\cP$--interpolating systems that will be useful in the sequel.
\begin{prop}\label{prop:P-interpolating_equivalent}
  Conditions~(\ref{item:non_empty})
  and~(\ref{item:zero_dim}) of Definition~\ref{def:P-interpolating}
  hold if and only if the following condition holds
  \begin{enumerate}[(i')]
  \item\label{item:ii'} There exist two sections
    $\overline A, \overline B$ of $\Gamma - \cP$ such that the
    supports of $\overline A$ and $\overline B$ have no common
    component.
  \end{enumerate}
\end{prop}

\begin{proof}
  Suppose (\ref{item:ii'}') holds. Then, the base locus of
  $\Gamma - \cP$ is contained in the intersection of the supports of
  $\overline A$ and $\overline B$. From Remark~\ref{rem:base_locus},
  this intersection contains $\cP$, then it is non empty and by
  (\ref{item:ii'}') has dimension $0$.

  Conversely, suppose that (\ref{item:non_empty}) and
  (\ref{item:zero_dim}) hold. From (\ref{item:non_empty}),
  there exists at least one section
  $\overline A$ of $\Gamma - \cP$. Let
  $\overline A_1, \dots, \overline A_s$ be the geometrically
  irreducible components of its support. Then, by
  (\ref{item:zero_dim}), for any $i \in \{1, \dots, s\}$, the linear
  system $\Gamma - \cP - \overline A_i$ is strictly contained in
  $\Gamma - \cP$. Consequently, consider the projective space $\P^r$
  parametrizing $\Gamma - \cP$, then the sublinear systems
  $\Gamma - \cP - \overline A_i$ correspond to finitely many of proper
  linear subvarieties of $\P^r(\Fqbar)$ and since $\Fqbar$ is
  infinite, there exists a point in $\P^r(\Fqbar)$ avoiding them all.
  Therefore, there exists a section $\overline B$ of $\Gamma - \cP$
  whose support has no common component with that of $\overline A$.
\end{proof}

\begin{rem}
  If $\Gamma$ is a sub--linear system of a linear system $\Delta$,
  then, $\Delta$ is $\cP$--interpolating too. On the other hand, for any
  $\cP' \subseteq \cP$, any $\cP$--interpolating linear system is
  $\cP'$--interpolating.
\end{rem}

The following statement is useful in what follows.

\begin{prop}\label{prop:minor_Gamma2}
  Let $\Gamma$ be a $\cP$--interpolating system, then
  \[
  \Gamma^2 \geq |\cP|.
  \]
\end{prop}

\begin{proof}
  From Proposition~\ref{prop:P-interpolating_equivalent}, there exist
  $\overline A, \overline B \in \Gamma - \cP$ with no common
  irreducible component.  The points
  of $\cP$ lying at the intersection of $\overline A$ and $\overline B$, we have thus
  $\Gamma^2 = \overline A \cdot \overline B \geq |\cP|.$
\end{proof}

\subsubsection{A lower bound for the minimum distance}

\begin{thm}\label{thm:bound}
Let $X$ be a smooth geometrically connected surface over $\F_q$ with
a set of rational points $\cP$ and
$G$ be a divisor on $X$ whose support avoids $\cP$.
Let $\Gamma$ be a $\cP$--interpolating linear system on $X$.
Then the minimum distance $d$ of $C_L (X, \cP, G)$ satisfies
\[
d\geq n-\Gamma \cdot G,
\]
where $n= |\cP|$.
\end{thm}

\begin{proof}
  Let $f$ be an element of $\H^{0}({X},{G})$ providing a nonzero
  codeword $\mathbf{c}$.
  Let $D$ be the positive divisor $D = (f) + G$. Then,
  \[  
  n- \wt{\textrm{ev}(f)}\leq |\suppt{D} \cap \cP|,
  \]
  where $\suppt{D}$ denotes the support of the divisor $D$.
  
  We claim that there exists $E\in \Gamma - \cP$
  whose support has no common irreducible component with that of $D$.
  Indeed, by Definition~\ref{def:P-interpolating}(\ref{item:non_empty}),
  $\Gamma - \cP$ is nonempty. Therefore it is parametrised by some
  projective space $\P^{\ell}$ for $\ell \geq 0$. Next, by
  Definition~\ref{def:P-interpolating}(\ref{item:zero_dim}), for any
  $\Fqbar$--irreducible component $\Ybar$ of the support of $D$,
  the linear system $\Gamma - (\cP \cup \Ybar)$ is distinct from
  $\Gamma - \cP$. Thus, there is a proper linear subvariety
  $V_{\Ybar}$ of $\P^{\ell}(\Fqbar)$ parametrising $\Gamma - (\cP \cup \Ybar)$.
  Let $\overline{W} \subseteq \P^{\ell}(\Fqbar)$ be defined as
  \[
  \overline{W} :=  \bigcup_{\Ybar} V_{\Ybar},
  \]
  where $\Ybar$ runs over all the $\Fqbar$--irreducible components of
  the support of $D$. Then, any element of
  $\P^\ell (\Fqbar) \setminus \overline{W}$ provides such a divisor
  $E$.  Consequently,
  \[
  | \suppt{D} \cap \cP| \leq |\suppt{D}\cap
  \suppt E| \leq D \cdot E=G \cdot \Gamma.
  \]
\end{proof}

Several examples of applications of this lower bound are
given further in \S~\ref{ss:universal} and \ref{ss:further_examples}.
Before, let us state a brief summary of our estimates.

\subsection{Summary on the parameters of codes
  for a surface}

Our previous results lead to the following statement.

\begin{thm}\label{thm:dim}
  Let $\Gamma$ be a $\cP$--interpolating linear system such that $n > \Gamma \cdot
  G$, and assume moreover that there exists an ample divisor $H$ on $X$
  such that $G\cdot H > K_X \cdot H$.
  Then the code $\AGCode{X}{\cP}{G}$ has parameters
  \[
  k \geq \frac 1 2 G \cdot (G-K_X) + \chi (\O_X) \qquad
  {\rm and} \qquad d \geq n - \Gamma \cdot G.
  \]
\end{thm}

\begin{proof}
  The condition $n > \Gamma \cdot G$ entails that the evaluation map
  is injective and hence asserts that the minimum distance of the code
  is that of the Riemann-Roch space.
\end{proof}

\subsection{A ``universal'' example of application of our
  bound}\label{ss:universal}
A natural question is : {\em how to find a $\cP$--interpolating system?}
The following lemma provides an example that can be applied to any surface
with a very ample divisor.

\begin{thm}\label{thm:universal}
  Let $\L$ be a very ample sheaf on $X$, then for any $\cP$ contained
  in $X(\F_q)$, the sheaf $\L^{\otimes (q+1)}$ is $\cP$--interpolating and, in
  the conditions of Theorem~\ref{thm:bound}, we get
\[
d\geq  n-(q+1)G \cdot \L.
\]
If in addition there exists $L \in |\L|$
such that $\cP$ is contained in the affine chart  $X \setminus \suppt{L}$, then
$\L^{\otimes q}$ is $\cP$--interpolating and
\[
d\geq n-qG \cdot \L.
\]
\end{thm}

\begin{rem}
In particular, if $X$ is embedded in $\P^\ell$, then
$\O_X(q+1)$ is $\cP$--interpolating.  
\end{rem}

The proof of Theorem~\ref{thm:universal} rests on the two following
well--known statements.

\begin{lem}
  \label{lem:arrangement_of_affine_hyperplanes}
  Let $x_1, \ldots, x_\ell$ be a system of coordinates in the affine space
  $\mathbf A^\ell$, then the intersection of the hypersurfaces of equation
  $x_i^q-x_i$ for $i = 1, \ldots, \ell$ equals $\mathbf A^\ell (\F_q)$.
\end{lem}

\begin{lem}\label{lem:arrangement_of_hyperplanes}
 Let $\ell$ be a positive integer and $x_0, \ldots, x_\ell$ be a system
 of homogeneous coordinates in $\P^\ell_{\F_q}$.
 For any pair $(i,j)$ with $0 \leq i < j \leq \ell$, let $\mathcal H_{ij}$
 be the hypersurface of equation $x_i^qx_j-x_ix_j^q$.
 Then, the variety
 $ \bigcap_{0 \leq i < j \leq \ell} \mathcal H_{ij}
 $
 has dimension $0$ and is the union of the rational points of $\P^\ell$.
\end{lem}

\begin{proof}[Proof of Theorem~\ref{thm:universal}]
Let $\phi : X \hookrightarrow \P^\ell$ with $\ell = h^0(X, \L)-1$
be the embedding associated to $\L$.
Hence $\L$ is isomorphic to $\phi^* \O_{\P^n}(1)$ and we denote
by $x_0, \ldots, x_\ell$ a
system of homogeneous coordinates of $\P^\ell$. 

The reduced base locus of $|\L^{\otimes (q+1)}| - \cP$ is contained in
the intersection of the zero loci of the sections
$(x_i^q - x_ix_j^{q-1})x_j$ of $\L^{\otimes (q+1)}$ for
$0 \leq i < j \leq \ell$.  From
Lemma~\ref{lem:arrangement_of_hyperplanes}, this zero locus equals
$X(\F_q)$. Therefore, the sheaf
$\L^{\otimes (q+1)}$ is $\cP$--interpolating.

For the second situation, we can apply
Lemma~\ref{lem:arrangement_of_affine_hyperplanes} to get the result.
\end{proof}

\subsection{Further examples}\label{ss:further_examples}

In \S~\ref{ss:universal} we proposed a generic linear system that is
$\cP$--interpolating for any set $\cP$ of rational points.  When we have
further information on the geometry of the surface it is possible to
construct more specific $\cP$--interpolating linear systems as suggested in
the examples to follow.

\subsubsection{Product of curves}
Let $X = C \times D$ be a product of curves and denote by
$\pi_C, \pi_D$ the corresponding projections.  Assume moreover that
$\cP$ is a ``grid'' of rational points, i.e. there is a set $\cP_C$ of
rational points of $C$ and a set $\cP_D$ of rational points of $D$
such that $\cP = \cP_C \times \cP_D$.  Note that $\cP_C$, $\cP_D$ can
be regarded as reduced positive divisors respectively of $C$ and
$D$. Assume moreover that these divisors are base point free. For
instance, assume that $\cP_C, \cP_D$ have respective degrees
$t_C, t_D$ larger than or equal to $2g_C$ and $2g_D$ where $g_C, g_D$
denote the respective genera of $C$ and $D$.  Then, the linear system
\[
\Gamma := |\pi_C^* (\cP_C) + \pi_D^* (\cP_D)|
\]
is $\cP$--interpolating. Indeed, by hypothesis, $\cP_C$ is base point free
and hence, there exists a positive divisor $E_C$ on $C$ (resp. $E_D$
on $D$) which is linearly equivalent to $\cP_C$ (resp $\cP_D$) and
with disjoint support.  Hence, the linear system $\Gamma$ contains
$\pi_C^* E_C + \pi_D^* \cP_D$ and $\pi_C^* \cP_C + \pi_D^* E_D$.  The
supports of these divisors have no common component.  We deduce from
Proposition~\ref{prop:P-interpolating_equivalent}  that $\Gamma$
is $\cP$--interpolating.

As a consequence, $\Gamma$
is numerically equivalent to
\[
  \Gamma \equiv t_C F_C + t_D F_D,
\]
where $F_C, F_D$ denote the respective numerical equivalences of a
fibre by $\pi_C$ and $\pi_D$ and $t_C, t_D$ denote the respective
degrees of the divisors $\cP_C, \cP_D$ respectively on $C$ and $D$.
Hence, given a divisor $G$ on $X$, the code
$\AGCode{X}{\cP}{G}$ has minimum distance
\[
  d \geq n - t_C F_C \cdot G - t_D F_D \cdot G
\]
by Theorem~\ref{thm:bound}.

\begin{exmp}\label{ex:P1xP1}
  On the product of two projective lines, the divisor class group is
  generated by two classes $H, V$. Let $G = aH + b V$.  For the choice
  of $\Gamma$, one can easily see that $H+V$ is very ample and hence,
  from Theorem~\ref{thm:universal}, the system
  $|(q+1)(H+V)|$ is $\cP$--interpolating. Therefore, we get
  \[
  d \geq n - (q+1)(a+b).
  \]
  Actually the exact minimum distance is known to
  be $n - (q+1)(a+b) + ab$ (see \citeMain[Theorem 2.1 \& Remark 2.2]{CouDuu}).
\end{exmp}

\subsubsection{Fibred surfaces}

Let $\pi : X \rightarrow C$ be a fibred surface over a curve
$C$ of genus $g$. Let $F_1, \ldots, F_r$ be the
fibres under $\pi$ of $r$ distinct rational points of $C$,
let $C_1, \ldots, C_s$
be $s$ distinct sections of $\pi$
and consider the set of rational points $\cP$ given by
\[
\cP = (F_1 \cup \cdots \cup F_r) \cap (C_1\cup \cdots \cup C_s).
\]
Suppose in addition that $r > 2g_C$ where $g_C$ denotes the genus of
$C$ and that there exists $C \in |C_1 + \cdots + C_s|$ whose support
avoids any element of $\cP$.  Then, the complete linear system
$|F_1+\cdots + F_r + C_1 + \cdots + C_s|$ is $\cP$--interpolating.

Indeed, similarly to the previous case, $F_1 + \cdots + F_r$ is the
pullback by $\pi$ of a base point free divisor on $C$ and there exists
a positive divisor $F \sim F_1 + \cdots + F_r$ whose support does not
contain any of the $F_i$'s. Next, the divisors
\[
F_1 + \cdots + F_r + C \qquad {\rm and}
\qquad F + C_1 + \cdots + C_s
\]
have no common component and are both in
$|F_1 + \cdots + F_r + C_1 + \cdots + C_s|$. Thus, according to
Proposition~\ref{prop:P-interpolating_equivalent}, the linear system $\Gamma$
is $\cP$--interpolating.

\subsubsection{Hirzebruch surfaces}
To get a more explicit example, consider the case of a rational ruled
surface, i.e. a Hirzebruch surface $\hirz{e}$. Such a surface is
ruled, i.e. there is a morphism $\pi_e : \hirz{e} \rightarrow \P^1$ with
a section whose image in $\hirz{e}$ is denoted by $S_e$.  This surface
has a discrete Picard group generated by $S_e$ and a fibre $F_e$ and
\[
F_e^2 = 0, \qquad F_e \cdot S_e = 1\qquad {\rm and}\qquad S_e^2 = -e.
\]

The surface can be obtained from $\P^2$ by a sequence of blow up and
blow down:
\begin{itemize}
\item $\hirz{1}$ is the blow up of $\P^2$ at a point $P$.
  Fix a line $L \subseteq \P^2$ containing $P$.  We
  denote by $S_1$ the exceptional divisor and by $F_1$ the
  strict transform of $L$. We have $S_1^2 = -1$ and
  $F_1^2 = 0$. 
\item $\hirz{e+1}$ is obtained from $\hirz{e}$ as follows. Let 
  $P$ be the point at the intersection of $F_e$
  and $S_e$. Blow up $\hirz{e}$
  at $P$ and denote by $E$ the exceptional divisor and by 
  $\widetilde S_e, \widetilde F_e$ the respective strict transforms
  of $S, F$ by the blowup map. Blow down $\widetilde F_e$ and 
  set $F_{e+1}$ and $S_{e+1}$ the respective images of $E$ and $\widetilde
  S_e$ by the blow down map. See Figure~\ref{fig:Sigma_ep1} for
  an illustration.
\end{itemize}

\begin{figure}[!h]
  \centering
  \includegraphics[scale = .35]{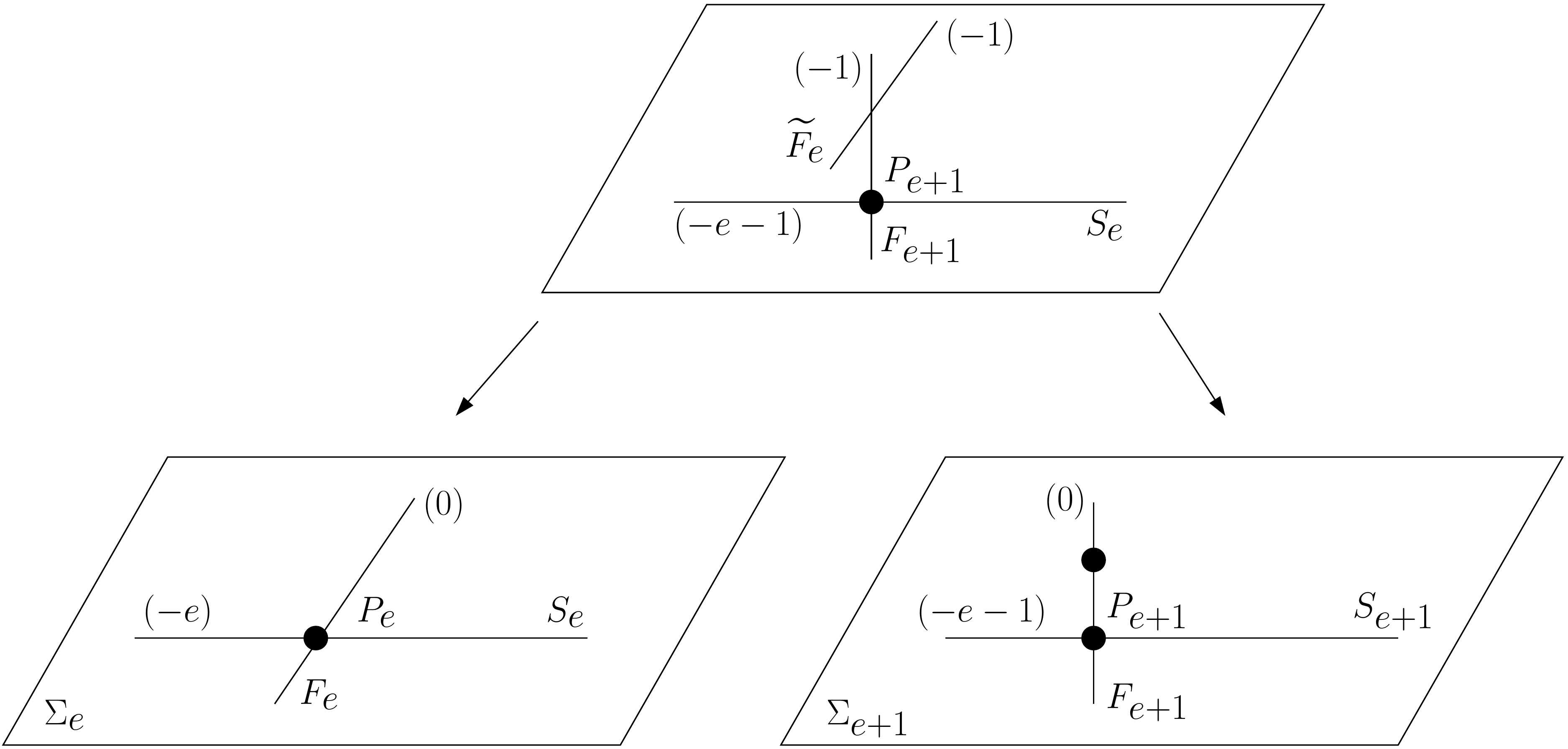}
  \caption{From $\hirz{e}$ to $\hirz{e+1}$}
  \label{fig:Sigma_ep1}
\end{figure}

Such surfaces are rational, in particular there exists a
birational map
\begin{equation}\label{eq:Hirz_birat}
\psi : \P^2 \stackrel{\sim}{\dashrightarrow} \hirz{e}
\end{equation}
which induces an isomorphism
between the affine chart $\P^2 \setminus L$ and the affine chart
$\hirz{e} \setminus (S_e \cup F_e)$.

Consider in $\P^2$ a line $L'$ distinct from $L$ and not containing
$P$. Let $Q$ be the point at the intersection of $L$ and $L'$. The
Zariski closure of the image of $L' \setminus \{Q\}$ by the map $\psi$
of (\ref{eq:Hirz_birat}) is a genus $0$ curve $C$ on $\hirz{e}$ which
is another section of $\pi_e$ which does not meet $S_e$. Therefore, it
satisfies
\[C \cdot S_e = 0 \qquad
  {\rm and} \qquad C \cdot F_e = 1.
\]
Thus,
\begin{equation}\label{eq:C}
C \sim eF_e + S_e.  
\end{equation}

Let $(X: Y: Z)$ be a system of homogeneous coordinates in $\P^2$ such
that $L = \{X = 0\}$, $P = (0:1:0)$ and $Q = (0:0:1)$.
Let $\cP_0$ be the set of rational points
$\{(1:x:y) ~|~ x\in A, \ y \in B\}$ for some sets $A \subseteq \F_q$
and $B \subseteq \F_q$ of respective cardinalities $a, b$. Finally,
let $\cP$ be the image of $\cP_0$ by $\psi$. The points in $\cP$
and $\cP_0$ are represented by filled black dots
\includegraphics[scale = .3]{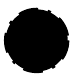}
in Figure~\ref{fig:psi}.

Note first that in $\P^2 \setminus L$ the set of points $\cP_0$
is contained in the following two unions of parallel lines:
\[
\bigcup_{t \in A} \{Y = tX\} \qquad {\rm and} \qquad \bigcup_{u\in B} \{X = uZ\}.
\]
On the left--hand part of Figure~\ref{fig:psi}, the first family of
lines is represented by \includegraphics[scale =
.4]{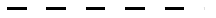} and the second one by
\includegraphics[scale = .4]{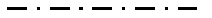} The Zariski closures of
the respective images of these lines by $\psi$ are
denoted as $V$ and $H$ and, from~(\ref{eq:C}), we have
\[
V \sim aC \sim a(eF_e + S_e) \qquad {\rm and} \qquad
H \sim b F_e.
\]

\begin{figure}[!h]
  \centering
  \includegraphics[scale=.45]{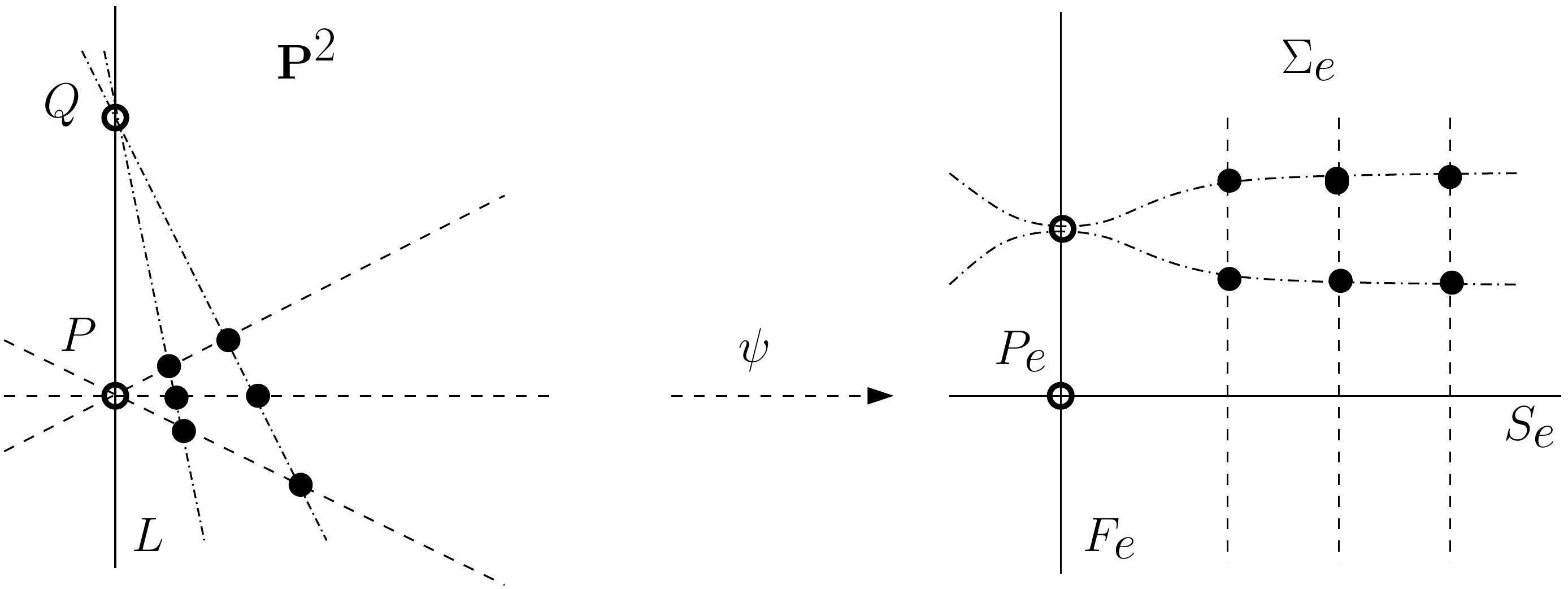}
  \caption{Behaviour of the family of lines under $\psi$}
  \label{fig:psi}
\end{figure}

We claim that the linear system $\Gamma := |(a+be)F_e + bS_e|$ is
$\cP$--interpolating. Indeed,
since $|F_e|$ is base point free, then so is $|bF_e|$ and there exists $H'$
in $|bF_e|$ such that $\cP \cap \suppt{V'} = \emptyset$. 
Similarly, the divisor $eF_e + S_e$ is very ample
\citeMain[Corollary 2.18]{Hartshorne} and hence the corresponding linear
system is base point free. Thus there exists $V' \sim V$
whose support avoids any point of $\cP$. Finally, the divisors 
\[
V+F' \qquad {\rm and } \qquad V'+F
\]
are both in $\Gamma$ and the intersection of their supports is
$\cP$. Therefore, $\Gamma$ is $\cP$--interpolating.

Now, consider a divisor $G \sim uF_e + vS_e$ on $\hirz{e}$. Using the previous
context, we get a code $\AGCode{\hirz{e}}{\cP}{G}$ with length
$n = ab$ and minimum distance
\[
  d \geq n - (a+be)v - ub + ebv.
\]

On the other hand if one wants to take $\cP = \hirz e (\Fq)$, that is
to say evaluating global sections at all the rational points, one can
use \S~\ref{ss:universal}. For this sake, one needs a very ample
divisor.  According to \citeMain[Corollary V.2.18(a)]{Hartshorne}, the
divisor $H = (e+1)F_e + S_e$ is very ample. Therefore, one ca take
$\Gamma = |(q+1)(e+1)F_e + (q+1)S_e|$. For a divisor
$G \sim uF_e +vS_e$ leads to the following lower bound for the minimum
distance.
\begin{align*}
  d &\geq \overbrace{(q+1)^2}^{=n} - (q+1)\left( (e+1)v + u -ev \right)\\
    &\geq (q+1)^2 - (q+1)(u+v).
\end{align*}

\begin{rem}
  The exact minimum distance and dimension of codes from Hirzebruch
  surfaces when the evaluation set is the full set of rational points
  have been computed by Nardi in \citeMain{nardi19}. According to this
  reference, when $v \geq 1$, the genuine minimum distance is
  $d = q(q-u+1)$. In particular, it does not depend on $v$. Here we can
  estimate the defect of our bound compare to the actual minimum
  distance :
  \[
    q(q-u+1) - (q+1)^2 + (q+1)(u+v) = u+v-1 + q(v-1).
  \]
  In particular, the larger the $v$, the worst our estimate.
\end{rem}

\subsection{Previous estimates for the minimum distance in the
  literature}\label{ss:other_estimates}

Let us first list some previous estimates given in the literature,
which can roughly speaking be divided into three categories.

\subsubsection{Using the maximum number of points of curves in a linear
  system: Aubry's bound}
The first category of bounds consists in observing that the minimum
distance is related to the maximum number of rational points of an
element of the linear system $|G|$, that is to say,
\[
d  \geq n - \max_{C \in |G|} |C(\F_q)|.
\]
The problem is then translated into that of getting an upper bound on
such a maximum. This is the approach used for instance by Aubry in
\citeMain{Aubry} who obtained the following result.

\begin{prop}[{\citeMain[Proposition 3.1(ii)]{Aubry1992}}]\label{prop:aubry}
  Let $D$ be a very ample divisor on a nonsingular projective surface $X$.
 Then the minimum distance of $\AGCode{X}{\cP}{D}$ satisfies
 \[
 d \geq n - D^2(q+1).
 \]
\end{prop}

When $G \sim \O_X(1)$, and taking as $\cP$--interpolating system
$\Gamma \sim \O_X(q+1)$ as suggested in \S~\ref{ss:universal}, our
bound yields the very same result as Aubry's bound.  The exemples to
follow show that our bound can actually be much better than Aubry's
one in other situations.  Another advantage of Theorem~\ref{thm:bound}
is that it does not require the very ampleness of the divisor $G$.

\begin{exmp}
  Consider the case of a surface $X$ together with a very ample sheaf
  $\O_X (1)$ and $G$ be a divisor such that $G \sim \mathcal{O}_X (d)$
  for some positive integer $d$. Thus, Aubry's result yields
 \begin{equation}
   \label{eq:aubry}
   d\geq n- (q+1)G^2= n-(q+1)d^2 \deg(X),
\end{equation}
while our bound yields
 \begin{equation}
   \label{eq:couvreur}
   d\geq n-G \cdot \mathcal{O}(q+1)=n - d(q+1)\deg (X).
 \end{equation}
Thus a $d^2$ is replaced by a $d$.  
\end{exmp}

\begin{exmp}
  Back to Example~\ref{ex:P1xP1}, for any positive integers $a, b$, the
  divisor $G = aH+bV$ is very ample and Aubry's bound yields
  \[
  d \geq n - 2ab(q+1).
  \]
  Here our ``$a+b$'' term replaces an ``$ab$'' term in the above
  lower bound.
\end{exmp}

\noindent {\bf Caution.} In the sequel, we discuss two lower bounds
for the minimum distance, that are due to S.~H. Hansen.  One
is presented in \S~\ref{ss:HansenA} and rests on the use of a set of
auxiliary curves. The second one is introduced in
\S~\ref{ss:HansenS} and involves the Seshardi constant.  To avoid
confusion, the first kind of bound will be referred to {\em Hansen
  bound (A)}, where (A) stands for {\em Auxiliary curves} and the
second one to {\em Hansen bound (S)}, where the (S) stands for
{\em Seshadri constant}.

\subsubsection{Using an auxilliary set of curves : Hansen's bound (A)}
\label{ss:HansenA}
The second category of estimates consists in using an auxilliary
family of irreducible curves $C_1, \ldots, C_s$ such that the curve
$C := C_1 \cup \cdots \cup C_s$ contains $\cP$.
Bounds of this kind appear in \citeMain[Theorem 4]{bouganis} and
\citeMain[\S 3.2]{soH}.
For instance, we have the following statement.

\begin{prop}[Hansen bound (A), {\citeMain[Proposition
    3.2]{soH}}]\label{prop:Hansen_avec_courbes_test}
  Let $X$ be a normal projective variety defined over $\F_q$ of
  dimension at least two. Let $C_1, \ldots, C_a$ be (irreducible)
  curves on $X$ and $\cP = \{P_1, \ldots, P_n\}$ be a set of
  $\F_q$--rational points of $X$.  Assume that for all
  $1 \leq i \leq a$, the number of $\F_q$--rational points on $C_i$ is
  less than an integer $N$.  Let $\L$ be a line bundle on $X$ over
  $\F_q$ such that $\L \cdot C_i \geq 0$ for all $i$. Let
  \[
  \ell := \sup_{s \in \H^0({X},{\L})} |\{i ~|~ Z(s)\ {\rm contains\ }C_i\}|,
  \]
  where $Z(s)$ denotes the vanishing locus of $s$. Then the code $\AGCode{X}{\cP}{\L}$
  has length $n$ and minimum distance
  \[
  d \geq n - \ell N - \sum_{i=1}^a \L \cdot C_i.
  \]
  Moreover, if $\L \cdot C_i = \eta \leq N$ for all $i$, then 
  \[
  d \geq n - \ell N - (a - \ell)\eta.
  \]
\end{prop}

\subsubsection{Using Seshadri constants, Hansen bound
  (S)}\label{ss:HansenS}
The third kind of estimate is based on Seshadri constants whose definition is
recalled below.

\begin{defn}[Seshadri constant]
  Let $\L$ be a line bundle on $X$ and $\mathcal P$ be a union of
  closed points of $X$.  Let
  $\pi : \textrm{Bl}_{\mathcal P} X \rightarrow X$ be the blowup of
  $X$ at $\mathcal P$.  Then the local Seshadri constant is defined as
  \[
  \sesh{X}{\cP}{ \L} := \sup \{\varepsilon \in \Q ~|~ \pi^* \L -
    \varepsilon E {\rm \ is \ nef}\},
  \]
  where $E$ denotes the exceptional divisor of $\textrm{Bl}_{\mathcal P}(X)$.
\end{defn}

\begin{prop}[Hansen bounds (S), {\citeMain[Proposition 3.1]{soH}}]\label{prop:Hansen_bound}
  Let $X$ be a smooth projective surface defined over $\F_q$.  Let
  $\L$ be a line bundle on $X$.
  \begin{enumerate}[(S1)]
  \item\label{item:Hansen_Seshadri} Suppose $\L$ is ample with Seshadri constant
    $\sesh{X}{\cP}{\L} \geq \varepsilon$. Then the corresponding
    code has minimum distance
    \[
    d \geq n - \frac{\L^2}{\varepsilon}\cdot
    \]
  \item\label{item:gen_by_global_sections}
      Let $\mathcal I$ be the ideal sheaf of the $\F_q$--rational
  points $\cP  = \{P_1, \ldots, P_n\}$.
  Suppose $\L^{\otimes \xi}
    \otimes \mathcal I$ is generated by global sections (such $\xi \in
    \mathbf N$ exists for instance if $\L$ is ample). Then,
    \[
    d \geq n - \xi \L^2.
    \]
  \end{enumerate}
\end{prop}

\begin{rem}
  Actually, in \citeMain{soH} the author states the result for codes on
  an arbitrary variety. 
\end{rem}

Note that using Proposition~\ref{prop:Hansen_avec_courbes_test}, Hansen
\citeMain{soH}
obtains a better lower bound, namely
\[
d \geq n - (q+1)(a+b) + ab
\]
which turns out to be the actual minimum distance as proved in 
\citeMain[Theorem 2.1 \& Remark 2.2]{CouDuu}.

\begin{rem}
  Actually, Hansen bound (S\ref{item:gen_by_global_sections}) is very
  close to Aubry's one. In particular, when $\L$ is a very ample line
  bundle on $X$, then $\L ^{\otimes (q+1)} \otimes \mathcal I$ is
  always generated by its global sections.  Indeed, consider the
  embedding $X \hookrightarrow \P^\ell$ associated to $\L$, using the
  notation of Lemma~\ref{lem:arrangement_of_hyperplanes} one can check
  that the global sections of the form $x_i^q x_j - x_i x_j^q$
  generate locally the sheaf $\L ^{\otimes (q+1)} \otimes \mathcal I$
  at any point.  Consequently, for the choice $\xi = q+1$, Hansen's
  bound of
  Proposition~\ref{prop:Hansen_bound}(\ref{item:gen_by_global_sections})
  is nothing but Aubry's bound.
\end{rem}

\subsection{Further discussion about Seshadri constants}
Our Theorem~\ref{thm:bound}, which provides a lower for the minimum
distance can be related with the Seshadri constant as follows.
  \begin{thm}
    Let $d^* = n - \Gamma \cdot G$ be the lower bound for the minimum distance
    given by Theorem~\ref{thm:bound} of $\AGCode{X}{\cP}{G}$. We have
    \[
    d^* \leq (1- \sesh{X}{\cP}{G})n.
    \]    
  \end{thm}

  \begin{proof}
  Let $a \in \Q$ such that $\pi^* G - a E$ is nef. Then,
  consider the strict transform $\widetilde{\Gamma}$ of $\Gamma$
  by the blowup map $\pi : \textrm{Bl}_{\mathcal P}(X) \rightarrow X$.
  We have
  \begin{align*}
    (\pi^* G - aE) \cdot \widetilde{\Gamma} & \geq 0,
  \end{align*}
  that is to say,
  \begin{align*}
    (\pi^* G - aE) \cdot (\pi^* \Gamma - \sum_{p\in \mathcal P}
    \textrm{mult}_P(\Gamma) E_P) & \geq 0,
  \end{align*}
  where
  $\textrm{mult}_P (\Gamma) = \min \{\textrm{mult}_P(C) ~|~ C\in
  \Gamma\}$.
  This leads to
  \begin{align*}
    \Gamma \cdot G - a\sum_{P \in \mathcal P} \textrm{mult}_P (\Gamma) \geq 0.
  \end{align*}
  By definition of a $\mathcal P$--interpolating system,
  $\textrm{mult}_P (\Gamma) \geq 1$ for all $P \in \mathcal P$.
  This entails
  $
  \Gamma \cdot G \geq  an.
  $
  This holds for any $a \in \Q$ such that $\pi^* G - a E$ is nef, and hence,
  we get
  \begin{equation}\label{eq:upper_bound_seshadri}
  \Gamma \cdot G \geq \sesh{X}{\cP}{G}n.
\end{equation}    
  \end{proof}

\begin{rem}
  Inequality~(\ref{eq:upper_bound_seshadri}) suggest another interesting
  application of $\cP$--inter\-po\-lating linear systems. They permit to get
  upper bounds for Seshadri constants.
\end{rem}

\subsection{Behaviour in towers}\label{s:behaviour_in_towers}

Another interest of our approach is that our criterion can be lifted
by a finite morphism. Hence it can be used to estimate the asymptotic
parameters of codes from towers of surfaces. 

\begin{prop}
  Let $\varphi: Y \rightarrow X$ be a finite morphism of smooth
  projective surfaces. Let $\cP$ be a set of rational points of $X$, and 
  $\cP_0$ be a set of rational points of
  $Y$ such that $\varphi (\cP_0)\subseteq \cP$. If $\Gamma$ is a
  $\cP$--interpolating  linear
  system on $X$, then $\varphi^{*} \Gamma$ is $\cP_0$--interpolating.
\end{prop}

\begin{proof}
  Let $H\in \Gamma - \cP$, then $\varphi^{*} H$ is in
  $\varphi^{*}\Gamma - \cP_0$ since $\varphi (\cP_0)\subseteq
  \cP$. Assume now that $\varphi^{*} \Gamma - \cP_0$ had a base
  curve and let $Y_0$ be an irreducible component of this base
  curve. This entails that $\varphi (Y_0)$ is in the base locus of
  $\Gamma - \cP$. Since by assumption this linear
  system has no base curve, $\varphi$ maps $Y$ into a single
  point. This contradicts the finiteness of $\varphi$.
\end{proof}

Note that no previous work in the literature investigated the
behaviour of a lower bound for the minimum distance in such a relative
case.  For this reason, and in order to push further the comparison
with other known bounds, we conclude this section by investigating the
behaviour of these bounds under morphisms.

\subsubsection{Aubry's bound} This bound is simple to apply
and hence could be used for estimates in towers. However, it requires
to have a very ample divisor at each level of the tower. Note that
even for a finite morphism, $\pi : Y \rightarrow X$, given a very
ample divisor $G$ on $X$, then, from \citeMain[Prop 5.1.12]{ega2},
$\pi^* G$ is ample on $Y$ but not necessarily very ample.

\subsubsection{Hansen bound (A)}
Contrarily to our bound, Hansen's bound (A) does not seem to be usable
for asymptotics. Even in the case of a single morphism
$\pi : Y \rightarrow X$. Let $\cP_X$ be a set of rational points of
$X$, $G$ be a divisor and set of curves $C_1, \ldots, C_a$ as in
Proposition~\ref{prop:Hansen_avec_courbes_test}.  In addition, suppose
there exists a set of rational points $\cP_Y$ of $Y$ such that
$\pi(\cP_Y) = \cP_X$. To study the code $\AGCode{P}{\cP_Y}{\pi^* G}$
one could consider the set of curves $\pi^* C_1, \ldots, \pi^* C_a$.
In the lower bound from
Proposition~\ref{prop:Hansen_avec_courbes_test},
\begin{align*}
d & \geq |\cP_Y| - \ell_Y N_Y - \sum_{i=1}^a \pi^* G \cdot \pi^* C_i\\
  & \geq |\cP_Y| - \ell_Y N_Y - (\deg \pi) \sum_{i=1}^a G \cdot  C_i.
\end{align*}
Thus, a part of the lower bound can be deduced from the lower
bound for the minimum distance of $\AGCode{X}{\cP}{G}$ and $\deg \pi$.
Even $N_Y$ can roughly be bounded above by $(\deg \pi)N$. However,
there does not seem to be a manner to deduce $\ell_Y$ from $\ell$.

\subsubsection{Hansen bounds (S)} Contrarily to the two previous ones,
the bounds given in Proposition~\ref{prop:Hansen_bound} behave well under
finite morphisms, and can be applied
to towers of finite morphisms. To prove it, first notice that the
criterion of
Proposition~\ref{prop:Hansen_bound}(\ref{item:Hansen_Seshadri})
rests on an assumption of ampleness. Fortunately, as we already
noticed, the pullback of an ample divisor by a finite map between
smooth surfaces is ample \citeMain[Prop 5.1.12]{ega2}. Next, to prove the
good behaviour of Hansen criteria in the relative case, we use the
following statements.

\begin{prop}
  Let $Y$ be a smooth projective geometrically connected surface over
  $\Fq$ and $\pi:Y \rightarrow X$ be a finite morphism. Then,
  \[
    \sesh{X}{\cP}{G} = \sesh{Y}{\cQ}{\pi^* G},
  \]
  where $\cQ$ is the set of points of $\Ybar$ defined as
  $\cQ := \pi^{-1}(\cP)$.
\end{prop}

\begin{proof}
  Consider the following diagram:
  \medskip

  \centerline{\xymatrix{\relax
      Y \ar[d]^{\pi} & \Bl{\cQ}{Y} \ar[l]_-{p_Y} \ar[d]^{\widetilde \pi} \\
      X              & \Bl{\cP}{X} \ar[l]^-{p_X}
    },
  }
  
  \medskip
  \noindent where $\Bl{\cP}{X}$ (resp. $\Bl{\cQ}{Y}$) denotes the
  blowup of $X$ at $\cP$ (resp. the blowup of $Y$ at $\cQ$). In
  addition, we denote respectively by $E_X$ and $E_Y$ the exceptional
  divisor of $\Bl{\cP}{X}$ and $\Bl{\cQ}{Y}$.  Let $s \in \Q$ be
  such that $p_X^* G - sE_X$ is nef. That is $s \leq \varepsilon (X, \cP, G)$.
  Let $C$ be a curve on $\Bl{\cQ}{Y}$ and suppose that
  \[
    C \cdot (p_Y^* \pi^* G - s E_Y) < 0.
  \]
  By the commutativity of the diagram, we get
  \[
    C \cdot (\widetilde \pi^* p_X^* G - s \widetilde \pi^* E_X) < 0.
  \]
  By the projection formula,
  \[
    {\widetilde \pi}_*C \cdot (p_X^* G - s E_X) < 0,
  \]
  which contradicts the original assumption on $s$.
  As a consequence, we get
  \[
    \sesh{X}{\cP}{G} \leq \sesh{Y}{\cQ}{\pi^* G}.
  \]

  Conversely, consider a rational number $s \leq \sesh{Y}{\cQ}{\pi^* G}$
  and a curve $D$ on $\Bl{\cP}{X}$ such that $D \cdot (p_X^* G - sE_X) < 0$.
  Then,
  \begin{align*}
    \widetilde \pi^* D \cdot (\widetilde \pi^* p_X^* G - s \widetilde
    \pi^* E_X) &< 0,\\
   \hbox{that is~} \widetilde \pi^* D \cdot ( p_Y^* \pi^* G - s
    E_Y) &< 0,
  \end{align*}
  which contradicts the original assumption on $s$ and concludes the
  proof.
\end{proof}

\begin{prop}
  Let $Y$ be a smooth projective geometrically connected surface over
  $\Fq$ and $\pi:Y \rightarrow X$ be a finite morphism. Let $\cQ$ be
  the set of points of $\Ybar$ defined as
  $\mathcal Q := \pi^{-1}(\cP)$. Let $\I_{\cP}$ (resp. $\I_{\cQ}$) be
  the ideal sheaf of $\O_X$ (resp. $\O_Y$) associated to the set of
  rational points $\cP$ (resp. $\cQ$).
  Suppose that the rational points in $\cP$ do not lie
  in the ramification locus of $\pi$. Let $\L$ be a line
  bundle on $X$ and $\xi$ a positive integer such that
  $\L^{\otimes \xi} \otimes \I_\cP$ is generated by its global
  sections, then
  $\pi^* \L^{\otimes \xi} \otimes \I_{\cQ}$ is generated by its global
  sections.
\end{prop}

\begin{proof} The sheaf $\L^{\otimes \xi} \otimes \I_\cP$ is generated by
  its global sections, which is equivalent to the existence of a surjective
  morphism
  \[
    \O_X^m \rightarrow \L^{\otimes \xi} \otimes \I_\cP
  \]
  for some positive integer $m$.  Since $\pi$ is finite, it is flat
  (\citeMain[Exercise 9.3(a)]{Hartshorne}) and hence the functor $\pi^*$ is
  exact. Thus, by pulling back by $\pi$ and using the exactness of
  $\pi^*$ we get a surjective morphism
  \[
    \O_Y^m \rightarrow \pi^* \L^{\otimes \xi} \otimes \pi^* \I_\cP
  \]
  and the non ramification hypothesis entails that $\pi^* \I_\cP = \I_\cQ$.
\end{proof}

Thanks to the previous results, we can assert that Hansen's bounds (S)
stated in Proposition~\ref{prop:Hansen_bound} applies in the relative
case. However, one needs to be careful about one fact: to get a code over
the same ground field, the set of points $\cQ = \pi^{-1}(\cP)$ should
contain only rational points. In addition, for
Proposition~\ref{prop:Hansen_bound}~(S\ref{item:gen_by_global_sections})
to hold, $\pi$ should not be ramified at the points of $\cP$ and
hence, $\cP$ should split completely.

\subsubsection{Summary of the behaviours of various bounds under
  finite morphisms}
Given a surface $X$ with a finite set of
rational points $\cP$ and a divisor $G$, an interesting question would be the following. Consider a tower of finite covers
\[
  \xymatrix{
    \relax
    \vdots \ar[d] \\
    X_n \ar[d]\\
    \vdots \ar[d]\\
    X_1 \ar[d]\\
    X
  }
\]
in which $\cP$ split totally. Is it possible to deduce lower bounds
for the parameters of the code
$\AGCode{X_n}{\pi_n^{-1}(\cP)}{\pi_n^* G}$, where $\pi_n$ denotes the
composed morphism $\pi_n : X_n \rightarrow X$, from some information
defined only on the bottom surface $X$?
\begin{itemize}
\item {\em Aubry's bound (Proposition~\ref{prop:aubry}).} Partially
  possible but a very ampleness condition on $\pi_n^* G$ to assert for
  every $n$;
\item {\em Hansen bound (A)
    (Proposition~\ref{prop:Hansen_avec_courbes_test}).} Seems
  difficult;
\item {\em Hansen bounds (S) (Proposition~\ref{prop:Hansen_bound}).}
  Yes : one only needs to know some lower bound for a Seshadri constant
  on the bottom surface $X$ or the integer $\xi$ such that
  $\L^{\otimes \xi} \otimes \I_\cP$ is globally generated.
\item {\em Our bound (Theorem~\ref{thm:bound}).} Yes, one only needs a
  $\cP$--interpolating linear system on the bottom surface $X$.
\end{itemize}

\subsection{How to get good towers of surfaces?}
This is clearly a natural question since codes from towers of curves
are well--known to provide the best known sequences of codes. In the
case of curves, one can see three approaches in the literature:
\begin{itemize}
\item modular curves, see for instance \citeMain{TVZ};
\item recursive towers of function fields, see for instance \citeMain{GS};
\item class field towers, see for instance \citeMain{LZ11}.
\end{itemize}

In \S~\ref{sec:philippe}, we investigate a class field like approach towards
by providing a criterion for a surface to have an infinite tower of étale
$\ell$--covers that splits totally above a fixed finite set of points.

\section{Infinite \'etale towers of surfaces}
In \S~\ref{sec:philippe} to follow, we investigate a class--field like
approach for the existence of an infinite tower of \'etale covers of a
given surface $X$. In particular, we prove
Theorem~\ref{Critere_Phi}, a sufficient condition for the
existence of such a tower in which some finite set $T$ of points
splits totally. 

This criterion raises the following question: {\em
  which explicit surfaces may have an infinite tower of \'etale covers
  in which some non-empty fixed set of rational points splits
  totally?}
For this sake, we prove in \S~\ref{sec:classification} that one can
find such an example on a given surface $X$ if, and only if, one can
find an example on some of its relatively minimal models. Then, we
consider the Kodaira classification of minimal surfaces, explaining
why some cases can quickly be excluded.  In particular, note that for
a surface to have an infinite étale tower that splits completely at
some finite set of rational points, a necessary condition is that its
geometric étale fundamental group is infinite.

We conclude this Section with an example in \S~\ref{sec:example} of
surface for which Theorem~\ref{Critere_Phi} can be entirely worked
out.
 
\subsection{\'Etale covers of marked surfaces} \label{sec:philippe}
\subsubsection{The \'etale site of a marked scheme}

Throughout this section, let $X$ be a scheme 
  and $T$ a
finite set of closed points.  For a morphism $\phi : Y\to X,$ we denote
by $T_{Y}$ the set $\phi^{-1}(T).$

\medskip

\noindent \textbf{Caution.} 
Compared to the previous section where the finite
set of points on the surfaces was referred to as $\cP$, in this
section we denote it by $T$ to be coherent with the usual notation and to emphasize that what follows holds
for a finite set of closed points and not only a set of rational
points.

\bigskip

\noindent {\bf Basic defintions and properties.}
In this paragraph, we briefly recall the definition of the marked
\'etale site of a marked scheme which has been introduced by
A. Schmidt \citeMain{schmidt2} in the case of curves and later
generalized to general schemes (see \citeMain{schmidt3}). A marked
scheme is a pair $(Y,S)$ where $Y$ is a scheme and $S$ is a set of
points of $Y.$ A morphism of marked schemes $f:(Z,R)\to (Y,S)$ is a
morphism of schemes $f:Z\to Y$ such that $f(R)\subset
S.$

We consider the marked scheme $(X,T)$ and we denote by $(X,T)_{\textrm{\'et}}$ its marked \'etale site. It consists in the category of morphisms $\phi : (U,S)\to (X,T)$ such that $\phi:U\to X$ is \'etale and $S=T_{U}$ (called $T$-marked \'etale morphisms) together with the following coverings:  surjective families $(p_{i}:(U_{i},S_{i})\to (U,S))_{i}$ such that, for any $s\in S,$ there is a $i$ and a $u_{i}\in S_{i}$ such that $p_{i}(u_{i})=s$ and  the induced field homomorphism $\kappa(s)\to\kappa(u_{i})$ is an isomorphism.

We can easily check that these data define a site (\textit{i.e.} a Grothendieck topology in
\citeMain[Def. 1.1.1]{Art}).
Thus, the
general properties of Grothendieck topologies apply here.

Let $P(X,T)$ and $S(X,T)$ denote respectively the category of
preshaves and sheaves of abelian groups on $(X,T)_{{\textrm{\'et}}}.$ An
example of such a sheaf is given by
$\mathcal{O}_{(X,T)}(U)=\Gamma(U,\mathcal{O}_{U}).$
We denote by $i:S(X,T)\to P(X,T)$ the inclusion functor and by
$\#:P(X,T)\to S(X,T)$ its left adjoint (see \citeMain[\S2.1]{Art} for
details of the construction). It is a general fact for Grothendieck
topologies that $i$ is left exact and $\#$ is exact (see
\citeMain[Th. 2.1.4]{Art}).

As usual, if $\pi:Y\to X$ is a morphism
and $\CG$ is a presheaf on $(Y,T_{Y})_{{\textrm{\'et}}},$ we define a
presheaf $\pi_{\ast}\CG$ by putting
$\pi_{\ast}\CG(U)=\CG(U\times_{X}Y)$ for any $T$-marked \'etale
morphism $U\to X.$ It is clear that it satisfies the axiom of being a
sheaf as soon as $\CG $ is a sheaf, and that the induced functor
$\pi_{\ast}:S(Y,T_{Y})\to S(X,T)$ is left exact.

We can define (see \citeMain[Th. 1.3.1]{Art}) a left adjoint $\pi^*$ to
$\pi_{*}.$ Indeed, for a presheaf $\CG$ on $(X,T)_{\textrm{\'et}},$ we define
first a presheaf $\pi'(\CG)$ on $(Y,T_{Y})$ by putting, for any
$T_{Y}$-marked \'etale map $V\to Y$, $\pi'(\CG)(V)=\varinjlim \CG(U),$
where the direct limit is taken on the $T$-marked \'etale maps
$U\to X$ making the diagram
\[\xymatrix{V\ar[d]\ar[r] & U\ar[d]\\ Y\ar[r] & X}\] commute.
If $\CF$ is
a sheaf, we define
$\pi^*(\CF)=\#\pi'(i(\CF)).$

As in any Grothendieck topology, $S(X,T)$ has enough injectives (see
\citeMain[Misc. 2.1.1]{Art} or \citeMain[Th. 1.10.1]{Gro}). For a sheaf
$\mathcal{F}$ of abelian groups on $(X,T)_{{\textrm{\'et}}},$ the functor
$\CF\mapsto \CF(X,T)=:\Gamma(X,T,\CF)$ is left exact and we denote by
$\H^i(X,T,\mathcal{F})$ its associated cohomogy groups. One can also
define the cohomology groups with support in a closed subscheme $Z,$
as the right derivatives $\H^i_{Z}(X,T,\mathcal{F})$ of the left exact
functor $\mathcal{F}\mapsto \ker (\CF(X,T)\to \CF(U,T\cap U)),$ where
$U$ is $X\setminus Z.$

As well as for the \'etale site, we can prove excision (see \citeMain[\S2--3]{schmidt3}) 
and a limit result, which is the main tool to
compute marked \'etale cohomology groups. For $x\in T,$ put
$X_{x}^h= \Spec(\O^h_{X,x})$, where $\O^h_{X,x}$ denotes the
henselianization of the local ring $\O_{X,x}$.  Let us
quote here the results of \citeMain{schmidt3} we will use.

\begin{prop} Let $X$ be a  scheme and $T$ a finite
  set of closed points.  Then, for every sheaf of abelian groups $\CF$
  on $(X,T),$ there is a long exact sequence
  \[\dots\to \H^i_{T}(X,T,\CF)\to \H_{\textrm{\'et}}^i(X,T,\CF)\to
  \H^i_{\textrm{\'et}}(X\setminus T,\CF)\to\dots\]
and, for all $i\geq 0,$
\[\H_{T}^i(X,T,\CF)\simeq \bigoplus_{x\in T}\H^i_{\{x\}}(X^h_{x},x,\CF).\] 
\end{prop}

Following Schmidt (see 
\citeMain[\S 5]{schmidt3}), one can define, for a noetherian, normal
and connected scheme $X$ a fundamental group $\pi_{1} (X,T)$ which is
profinite and classifies \'etale covers of $X$ where the points of $T$
split completely. Be cautious that fundamental groups are defined in
\citeMain{schmidt3} for more general schemes and that what we denote
here by $\pi_{1}(X,T)$ for convenience is denoted by
$\hat{\pi}_{1}^{\text{et}}(X,T)$ in 
\citeMain[\S 5]{schmidt3}(in our setting, there is no possible
confusion, thanks to \citeMain[Proposition 5.1]{schmidt3}).

We denote by $\widetilde{(X,T)}(\ell)$ the universal pro-$\ell$-cover
of $(X,T).$ The projection $\widetilde{(X,T)}(\ell)\to X$ is Galois
with Galois group the maximal pro-$\ell$-quotient $\pi_{1}(X,T)(\ell)$
of $\pi_{1}(X,T).$ Finally, the Hochschild-Serre spectral sequence
relates Galois cohomology and \'etale cohomology groups (see
\citeMain[\S 4--5]{schmidt3}). Let us also recall it here.

\begin{prop} Let $X$ be a connected noetherian scheme, $T$ a finite
  set of closed points of $X.$ Let $M$ be a discrete $\ell$-torsion
  $\pi_{1}(X,T)(\ell)$-module. Then we have the following spectral
  sequence:
\[E_2^{pq}=\H^p(\pi_{1}(X,T)(\ell),
\H^q(\widetilde{(X,T)}(\ell),T,M))\Rightarrow \H^{p+q}(X,T,M).\]
\end{prop}

It induces a five term exact sequence:
\begin{align*}0\to \H^1(\pi_{1}(X,T)(\ell),M)\to &\H^1_{\textrm{\'et}}(X,T,M)\to
  \H^1(\widetilde{(X,T)}(\ell),M)^{\pi_{1}(X,T)(\ell)} \to \\ &\to
  \H^2(\pi_{1}(X,T)(\ell),M)\to \H^2_{\textrm{\'et}}(X,T,M) .
\end{align*}
As $\H^1(\widetilde{(X,T)}(\ell),M)=0,$ because
$\widetilde{(X,T)}(\ell)$ admits no non-trivial $\ell$-cover, we see
that we have the following isomorphisms. 
\begin{cor} \label{HS} In the setting of the proposition, the
  Hochschild-Serre spectral sequence induces isomorphisms
\[\H^i(\pi_{1}(X,T)(\ell),M)\simeq \H^i_{\textrm{\'et}}(X,T,M) \text{ for } i=0,1\] and
an injection
\[\H^2(\pi_{1}(X,T)(\ell),M)\hookrightarrow
\H^2_{\textrm{\'et}}(X,T,M).\]
\end{cor}

\subsubsection{The cohomology of marked surfaces}
Let $X$ be a $2$-dimensional noetherian regular scheme defined over
$\F_{q}$ and let $T$ be a finite set of closed points. Let
$\Lambda=\F_{\ell},$ where $\ell\nmid q$ is prime number.  The aim of
this paragraph is to computate the \'etale cohomology groups
$\H^i(X,T,\Lambda)$ of the marked surface $(X,T).$ Notice that this can
be immediately generalized to higher dimensions.

\bigskip

\noindent {\bf Local computations.}
Let $\hX$ denote the spectrum $\spec ({\mathcal{O}_{X,x}^h})$ of the
henselisation of the local ring $\mathcal{O}_{X,x}$ of $X$ at a point
$x\in T$ with finite residue field $k.$ The scheme $\hX$ is henselian
with closed point $x,$ and therefore
\[\H^i_{\textrm{\'et}}(\hX,\Lambda)=
\H^i_{\textrm{\'et}}(x,\Lambda)=\H^i(k,\Lambda)\simeq\begin{cases}
\Lambda &\text{ for } i=0,1\\
0 &\text{ for } i\geq 2.\end{cases}\]

Let $\hU=\hX-x.$ We wish to understand the groups 
\[\H^{i}_{x}(\hX,\Lambda)\text{ and } \H^{i}_{x}(\hX,x,\Lambda).\]
The first groups arise in the sequence 
\[\cdots \to \H^{i}_{x}(\hX,\Lambda)\to \H^{i}(\hX,\Lambda)\to
\H^{i}(\hU,\Lambda)\to \cdots.\]
They are computed thanks to Gabber purity theorem \citeMain{fujiwara}: 
\[\H_{x}^i(\hX,\Lambda)\simeq \H^{i-4}(x,\Lambda(-2))\] as $x$ is regular of
codimension $2$ in the regular affine scheme $\hX.$ Then, they are $0$
for $i\neq 4,5$ and
$\H_{x}^4(\hX,\Lambda)\simeq \Hom(\mu_{\ell}(k)\otimes\mu_{\ell}(k),
\Lambda),$ $\H_{x}^5(\hX,\Lambda)\simeq \H^{1}(k,\Lambda(-2))$ this
latter having the same cardinality as
$\Hom(\mu_{\ell}(k)\otimes\mu_{\ell}(k),\Lambda)$ (thus
$\Lambda$-vector spaces of dimension $1$ for $i=4,5$ if $\ell$ divides
$q-1$ and trivial otherwise).

We deduce that $\H^1(\hU,\Lambda)\simeq \H^1(\hX,\Lambda),$ that
\[ \H^i(\hU,\Lambda)\simeq \H_{x}^{i+1}(\hX,\Lambda)\text{ for } i=3,4\]
and that $\H^i(\hU,\Lambda)=0$ for $i=2$ or $i\geq 5.$

The second groups appear in
\[\cdots \to \H^{i}_{x}(\hX,x,\Lambda)\to \H^{i}(\hX,x,\Lambda)\to
\H^{i}(\hU,\Lambda)\to \cdots.\] 

Since the identity of $(\hX,x)$ is cofinal among the covering families
of $(\hX,x),$ we have $\H^i(\hX,x,\Lambda)=0$ for $i\geq 1.$ Indeed, if
$f: Y\to \hX$ is a finite \'etale morphism marked at $x,$ then it
admits a section (see \citeMain[Th. I.4.2.]{MilEC}) inducing an
isomorphism between $\hX$ and a connected component of $Y$ (as $f$ is
separated). If $Y$ is connected then the map $f$ is an
isomorphism.

It implies
that
$\H^{1}_{x}(\hX,x,\Lambda)=\ker (\H^1(\hX,x,\Lambda)\to
\H^1(U_{x},\Lambda))=0$
and that $\H^{i}_{x}(\hX,x,\Lambda)\simeq \H^{i-1}(\hU,\Lambda)$ for any
$i\geq 2.$ Therefore, we have shown the following:
\begin{prop} Let $T$ be $\{x\}$ or $\emptyset.$ The groups
  $\H^i_{x}(\hX,T,\Lambda)$ are trivial for $i\leq 1,$ $i=3$ or
  $i\geq 6.$ Moreover, we have:
 \[\dim \H^i_{x}(\hX,T,\Lambda)=\begin{cases} \sharp T &\text{ if } i=2\\
    1 &\text{ if } i=4,5 \text{ and } \ell \mid q-1\\
    0 & otherwise
 \end{cases}\]
\end{prop}

\bigskip

\noindent {\bf Global computations.}
Let us now use excision and the local computations in order to get
information on the cohomology groups. Put $h^i(X,\dots)$ for the
dimension $\dim_{\Lambda}\H^i(X,\dots).$ As soon as it is well defined, put $\displaystyle \chi(X,\dots):=\sum_{i\geq 0}(-1)^ih^i(X,\dots).$

\begin{prop} Let $X$ be a $2$-dimensional noetherian regular scheme
  defined over $\F_{q}$ and let $T$ be a finite set of closed
  points. Let $\Lambda=\F_{\ell}.$
  \[h^2(X,T,\Lambda)-h^1(X,T,\Lambda)=\sharp T+h^2(X,\Lambda)-
  h^1(X,\Lambda)=\dim \H^2(\Xbar,\Lambda)^{G_{\F_q}}+\sharp T-1,\]
and
\[h^5(X,T,\Lambda)-h^4(X,T,\Lambda)+h^3(X,T,\Lambda)=
h^5(X,\Lambda)-h^4(X,\Lambda)+h^3(X,\Lambda),\]
where $G_{\F_q}=\mathrm{Gal}(\Fqbar, \F_q)$ and
$\Xbar=X\times_{\F_q} \Fqbar$.
\end{prop}
\begin{proof}
  We use excision to compute the $\H^i(X,T,\Lambda).$ It provides exact
  sequences
  \[\cdots \to \bigoplus_{x\in T} \H_{x}^i(\hX,\Lambda)\to \H^i(X,\Lambda)
  \to \H^i(X-T,\Lambda)\to \cdots \text{ and }\]
  \[\cdots \to \bigoplus_{x\in T} \H_{x}^i(\hX,T_{x},\Lambda)\to
  \H^i(X,T,\Lambda)
  \to \H^i(X-T,\Lambda)\to \cdots.\] The first sequence leads to
  isomorphisms $\H^i(X,\Lambda) \simeq \H^i(X-T,\Lambda)$ for $i=1,2,$
  and from the second one can deduce the first equality of the
  proposition by the local computations.

Moreover, the Hochschild-Serre spectral sequence leads to the
following, for any $i\geq 1:$
\[0\to \H^1(G_{\F_q},\H^{i-1}(\Xbar,\Lambda))\to \H^i(X,\Lambda)\to
\H^i(\Xbar,\Lambda)^{G_{\F_q}}\to 0.\] 
Together with the fact that for any finite module $M,$
$\H^0(G_{\F_{q}},M)$ and $\H^1(G_{\F_{q}},M)$ have same cardinality, we
obtain
\[h^2(X,\Lambda)-h^1(X,\Lambda)=\dim
\H^2(\Xbar,\Lambda)^{G_{\F_q}}-1,\]
which leads to
\[h^2(X,T,\Lambda)-h^1(X,T,\Lambda)=\dim
\H^2(\Xbar,\Lambda)^{G_{\F_q}}+\sharp T-1.\]
\end{proof}

The second part of the theorem comes from the isomorphisms
$\H_{x}^i(\hX,{x},\Lambda)\simeq \H_{x}^i(\hX,\Lambda)$ for $i\geq 3,$
and their triviality for $i=3$ and $6.$ Note that if $\ell\nmid q-1,$
then the exact sequences imply that
$\H^i(X,T,\Lambda)\simeq \H_{\textrm{\'et}}^i(X,\Lambda)$ for $i\geq 3.$

\begin{cor} Let $X$ be a smooth projective surface defined over
  $\F_{q}.$ Let $T$ be a finite set of closed points of $X.$ Then, the
  cohomology groups $\H^i(X,T,\F_{\ell})$ vanish for $i\geq 6,$ and we
  have $\chi(X,T,\F_{\ell})=\sharp T.$
\end{cor}

\begin{proof} Indeed, we have for $i\geq 6,$
  $\H^i(X,T,\F_{\ell})\simeq \H^i(X-T,\F_{\ell})\simeq
  \H^i(X,\F_{\ell})=0,$ as the local cohomology groups vanish, as $X$
  is smooth. Moreover, since $X$ is projective, $\chi(X,\F_{\ell})$
  vanishes by the Poincar\'e duality, which gives the desired result
  by the proposition.
\end{proof}

\bigskip

\noindent {\bf Choice of the closed points.}
 
Let $X$ be a smooth projective (absolutely irreducible) surface
defined over $\F_{q}$ and $T$ be a finite set of closed points
containing at least a point $x_{0}.$ Class field theory (see
\citeMain[VI.\S4--5]{ser} or \citeMain{Sam}) for a concise exposition) defines
a reciprocity map $\rho$ from the Chow group $\CH_{0}(X)$ of
zero-cycles of $X$ to $\pi^{ab}_{1}(X),$ by sending closed points $x$
to their Frobenius (the image in $\pi_{1}(X)^{ab}$ of the Frobenius of
$Gal(\overline{\kappa(x)}/\kappa(x))^{ab}).$ Its image is exactly the
set $\pi^{ab,f}_{1}(X)$ of elements of $\pi^{ab}_{1}(X)$ whose image
in $Gal(\overline{\F_{q}}/\F_{q})$ by the natural map is an element of
$\ZZ$ (in $\hat{\ZZ}).$ By \citeMain[VI.n$^\circ$ 16 Th\'eor\`eme 1]{ser}), we
have the following isomorphism of short exact sequences:

\medskip

\noindent \centerline{ \xymatrix{\relax
    0 \ar[r]& CH_{0}^0(X)\ar[r] \ar[d]_\rho^{\simeq}\commutatif & CH_{0}(X) \ar[r]^-{deg}  \ar[d]^\rho_{\simeq}\commutatif&\ZZ\ar[r]\ar[d]^{=} &0 \\
    0 \ar[r]&  \pi_{1}^{ab,0}(X)\ar[r] & \pi^{ab,f}_{1}(X) \ar[r]&\ZZ\ar[r] &0 \\
  } }

\medskip

The map $\rho$ induces an isomorphism from $\CH_0(X)/(x_{0})$ to
$\pi_{1}(X)^{ab}/(\rho(x_{0})).$ Then the group
$\pi_{1}(X)^{ab}/(\rho(x_{0}))$ is the finite quotient of
$\pi_{1}(X)^{ab}$ corresponding to the maximal \'etale abelian cover of $X,$
totally split at $x_{0}$ (see \citeMain[VI.\S5]{ser} : a finite \'etale
cover is totally split at $x_{0}$ if and only if its associated
Frobenius is trivial).  Therefore, we have an isomorphism
$\CH_0(X)/(\ell,x_{0})\simeq \pi_{1}(X,\{x_{0}\})^{ab}/\ell.$

Similarly, adding one more
point $x$ to $T$ lowers $h^1$ by one or enlarge $h^2$ by one,
depending on the fact that $x$ belongs to the group generated by the
previous split points in $\CH_{0}(X)/\ell$ or not.

Therefore, we have $h^1(X,T,\Lambda)=h^1(X)-r_{T},$ where $r_{T}$ is
the $\F_{\ell}$-dimension of the space generated by the points of $T$ in
$\CH_{0}(X)/\ell.$ If $T\neq \emptyset,$ $1\leq r_{T}\leq \sharp T.$

\subsubsection{Golod-Shafarevich criterion} 

Our strategy to produce infinite covers of $X$ with splitting
properties is the Golod-Shafarevich criterion. For this purpose, we
need the following inequality:
\[h^2(\pi_{1}(X,T)(\ell),\ZZ/\ell\ZZ)\leq \frac{h^1(\pi_{1}(X,T)(\ell),
  \ZZ/\ell\ZZ)^2}{4}\cdot\]
Because of the previous comparison maps (see Corollary \ref{HS}), it
is sufficient to check that it is satisfied for the \'etale cohomology
groups of the marked surface $(X,T).$

\begin{thm}\label{Critere_Phi}
  If
  \[\ds\dim \H^1(\Xbar,\Lambda)^{G_{\F_q}}\geq
    r_T+1+2\sqrt{\dim \H^2(\Xbar,\Lambda)^{G_{\F_q}}+\sharp T},
  \]
  then $\pi_{1}(X,T)(\ell)$ is infinite, i.e. $X$ admits a $T-\ell$
  infinite classfield tower.
\end{thm}
\begin{proof}
  The group $\pi_{1}(X,T)(\ell)$ is infinite if
  \[
    h^2(\pi_{1}(X,T)(\ell),\ZZ/\ell\ZZ)\leq
    \frac{h^1(\pi_{1}(X,T)(\ell),\ZZ/\ell\ZZ)^2}{4}\cdot
  \]
  As
  $h^1(\pi_{1}(X,T)(\ell),\ZZ/\ell\ZZ)=h^1_{\textrm{\'et}}(X,T;\Lambda):=h^1_{T}$
  and
  \[h^2(\pi_{1}(X,T)(\ell),\ZZ/\ell\ZZ)\leq
  h^2(X,T,\ZZ/\ell\ZZ)=:h^2_{T}\]
  because of Corollary \ref{HS}.  Therefore it sufficies to have
  $4h^2_{T}\leq (h^1_{T})^2.$

  Let $\bar{h}^i:=\dim \H^{i}(\Xbar,\Lambda)$ and
  $\bar{h}^{i,G}:=\dim \H^{i}(\Xbar,\Lambda)^{G_{\F_{q}}}.$ The
  previous inequality is equivalent to
  \[(h^1_{T}-2)^2\geq 4({h}^{2}_{T}-{h}^{1}_{T}+1)\] and
  to \[(\bar{h}^{1,G}-r_{T}-1)^2\geq 4(\bar{h}^{2,G}+\sharp T)\] and the
  proposition follows.
\end{proof}

The next proposition seems to be more useful in a computing point of
view, as the Euler-Poincar\'e characteristic is computable. Let
$\alpha=\bar{h}^1-\bar{h}^{1,G}$ and
$\bar{\chi}=\chi(\Xbar,\F_{\ell})=\chi(\Xbar,\Q_{\ell})$
(cf. \citeMain{MilEC} p166).
\begin{prop}
  The conclusion also holds if
  \[(\bar{h}^1-
  \alpha+r_{T}-5)^2\geq 4({\bar{\chi}+2\alpha+2r_{T}+4+\sharp T}).\] 
  If \end{prop}
\begin{proof} We introduce $\bar{\chi}=\bar{h}^{2}-2\bar{h}^{1}+2$ in
  what preceeds and note that $\bar{h}^{2,G}\leq \bar{h}^2.$
\end{proof}
Remark that the first $\ell$-adic Betti number $b_{1}$ of $\Xbar$
satisfies $b_{1}\leq \bar{h}^1.$ We may hope to be able to compute
$b_{1}$ and $\bar{\chi}$ by computing the Zeta function of ${X},$ and
thus to verify whether the previous inequality holds for $b_{1},$ in
particular after extending the constants so that $\alpha=0.$ When
$\bar{\chi}$ is very small 
the condition should be satisfied.

\subsection{Using the classification of surfaces}\label{sec:classification}
Theorem~\ref{Critere_Phi} raises the question: {\em which surfaces
  may have an infinite tower of étale covers in which some finite set
  $T$ splits completely?}  In the sequel, we sieve the Kodaiara
classification of surfaces in order to study which classes should be
excluded and which ones may provide good candidates.  The following
statement will be useful to exclude some cases.

\begin{prop}\label{prop:finiteness_of_pi_1}
  Let $\pi : \Ybar \rightarrow \Xbar$ be a finite morphism of smooth
  connected projective surfaces over an algebraically closed field. Suppose
  that $\Ybar$ has a finite fundamental group, then $\Xbar$
  cannot have an infinite tower of étale covers.
\end{prop}

\begin{proof}
  Since $\Ybar$ has a finite fundamental group, then its universal
  cover $\widetilde{\Ybar} \rightarrow \Ybar$ is a finite map
  and hence, replacing $\Ybar$ by its universal cover, one can
  suppose that $\Ybar$ is simply connected.

  Now, denote by $s$ the degree of the morphism
  $\Ybar \rightarrow \Xbar$ and consider an étale cover
  $\Xbar_1 \rightarrow \Xbar$ of degree $m$. Its pullback on $\Ybar$
  provides an étale cover of $\Ybar$ and since $\Ybar$ is supposed to
  be simply connected, this cover is the disjoint union of $m$ copies
  of $\Ybar$
  
  \noindent \centerline{ \xymatrix{ \relax
      \Ybar  \ar[d] & \bigsqcup_{i=1}^m \Ybar \ar[d]^-g \ar[l]_-{f} \\
      \Xbar & \Xbar_1 \ar[l] } } 
  and the restriction of $f$ to any
  component is an isomorphism.  Consider the
  restriction of $g$ to a connected component of
  $\bigsqcup_{i=1}^m \Ybar$, then the map $\Ybar \rightarrow \Xbar$ is
  factorised by a dominant map $\Ybar \rightarrow \Xbar_1$.  As a
  consequence the degree $s$ of $\Ybar \rightarrow \Xbar$ divides $m$
  and hence any finite étale cover of $\Xbar$ has a bounded degree.
\end{proof}

\subsubsection{Relatively minimal models}

Recall that a smooth surface $X$ over a field $k$ is said to be {\em
  relatively minimal} if there is no $(-1)$--curve of genus $0$ in
$\Xbar$. Equivalently, for such a surface, there does
not exist a regular surface $\Ybar$ over $\bar k$ and a $\bar k$--morphism
$\Xbar \rightarrow \Ybar$ which consists in blowing down a
finite set of curves in $\Xbar$.

The following result asserts that any \'etale cover of a surface comes
from an \'etale cover of a relatively minimal model, so that it is
sufficient to investigate relatively minimal surfaces. It is a
refinement for relatively minimal models of surfaces of the well known
birational invariance property of the fundamental group.  This result
is probably well--known by the experts, we give a proof because of
a lack of a reference.

\begin{thm}\label{pullback_relmin_model}
  Let $\Xbar_0$ be a surface over an algebraically closed field
  $\bar k$ and $\pi : \Xbar \rightarrow \Xbar_0$ be a surface obtained
  from $\Xbar_0$ after blowing up a point. Then for any \'etale cover
  $ f : \Ybar \rightarrow \Xbar$, there exists an \'etale cover
  $f_0 : \Ybar_0 \rightarrow \Xbar_0$ from a smooth surface $\Ybar_0$, such
  that $\Ybar = \Ybar_0 \times_{\Xbar_0}\Xbar$.
\end{thm}

\begin{proof}
From Stein Factorization Theorem
   \citeMain[Corollary III.11.5]{Hartshorne}, the map
   $f \circ \pi$ factorizes as
    \begin{center}
     \centerline{ \xymatrix{
        \relax
        \Ybar \ar[r]^f \ar[d]_{\pi'} & \Xbar \ar[d]^{\pi} \\
        \Ybar_0 \ar[r]^{f_0} & \Xbar_0
        }}
    \end{center}
 where $\pi'$ has connected fibres and $f_0$ is
 finite. 
 
 Let us prove that $f_0 : \Ybar_0 \rightarrow \Xbar_0$ is \'etale. For
 any variety $V$ in this diagram, we denote by $U_V$ the open
 subvariety of $V$ above the open subsheme $U_{\Xbar_0}$ of $\Xbar_0$
 obtained by puncturing the blown-up point $x_0$.  Let
 $\imath : \Cbar \hookrightarrow \Xbar_0$ be a geometric curve on
 $\Xbar_0$, and $U_{\Cbar}=\imath^* U_{\Xbar_0}$.  Since $f$ is
 \'etale, $\pi'^*\circ f_0^*U_{\Cbar}= f^*\circ \pi^*U_{\Cbar}$ is
 reduced, so that $f_0^*U_{\Cbar}$ itself is reduced and $f_0$ is
 \'etale outside $x_0$. By Zarisky purity Theorem~\citeMain{zariski},
 it follows that $f_0$ is \'etale.  In particular, $\Ybar_0$ is smooth
 since $\Xbar_0$ is.

By the universal property of
fiber products, there exists a unique morphism $\varphi$ such that the
following diagram commutes.
\begin{center}
  \centerline{ \xymatrix{ \relax \Ybar \ar@/^1pc/[rrd]^{f}
      \ar@{-->}[rd]^{\varphi} \ar@/_1pc/[rdd]_{\pi'}
      & & \\
      & \Ybar_0 \times_{\Xbar_0} \Xbar \ar[r]^-{p_2} \ar[d]_{p_1} &
      \Xbar \ar[d]^{\pi}\\
      & \Ybar_0 \ar[r]_{f_0} & \Xbar_0 } }
\end{center}
 Let us prove that $\varphi$ is an isomorphism, concluding the proof.
 Remind that $f=p_2 \circ \varphi$ is \'etale and $p_2$ is \'etale as
 a base change of the \'etale map $f_0$ \citeMain[Proposition
 I.3.3.(b)]{MilEC}. Thus, from \citeMain[Corollary I.3.6]{MilEC}, the
 morphism $\varphi$ is \'etale too. Therefore, to prove that $\varphi$
 is an isomorphism it is sufficient to prove that any geometric closed
 point $P = (y_0, x) \in \Ybar_0 \times_{\Xbar_0} \Xbar$ has a unique
 pre-image.  By the commutativity of the diagram, the pullback of $P$
 by $\varphi$ is contained in ${\pi'}^*(y_0)$, which is connected by
 Stein Theorem.  In the same way, it is contained in $f^*(x)$ which is
 finite since $f$ is finite.  Consequently, $P$ has a single inverse
 image by $\varphi$ and hence $\varphi$ is an isomorphism.
\end{proof}

\subsubsection{Classification of surfaces}
It follows from Theorem~\ref{pullback_relmin_model} that a surface $X$
admits an infinite tower in which a non-empty set $T$ of rational
points splits if and only if any relatively minimal model
$\pi : X\rightarrow X_0$ of $X$ also admits such an infinite tower, in
which $\pi(T)$ splits. This leads us to investigate the Kodaira
classification of minimal surfaces given in Table~\ref{tab:classif}
(see for instance \citeMain[\S II.8.1]{Encyclo_Alg_Geo_II} for the
characteristic 0 case and or \citeMain{Badescu} for the general case), in
which the numerical equivalence of divisors is denoted by
$\equiv$. For each class of surface in the table, we ask whether
such a tower can exists.  We indicate in a last column the conclusion
of the discussion that follows.

\medskip

\begin{table}[!h]
  \center
  \caption{Classification of algebraic surfaces}
  \label{tab:classif}
  \begin{tabular}{|c|c|c|c|c|c|}
    \hline
    Structure & $\kappa$ & canonical & $K^2$ & existence of infinite\\
              & &  class $K$ & & splitting tower\\
    \hline
    Rational& $ $ & - & - & no\\
              &$\quad -1 \quad$&&&\\
    Ruled above a& & - & $8(1-g)$& comes from the\\
    curve of genus  $g$& & & &base curve \\
    \hline
    Abelian&&&&no\\
    $K3$ &&&&no\\
    Enriques &&&&no\\ 
    Hyperelliptic &&&&no\\
                  & $\quad 0 \quad$  & $K \equiv 0$ & $K^2=0$&\\
    Quasi--elliptic (only&&&&?\\
	 in char.  $2$ and $3$)& &&&\\
    \hline
    Elliptic&$\quad 1 \quad$ & $\forall n>0, nK \neq 0$ &$K^2=0$&?\\
    \hline
    General type & $\quad 2 \quad$ & $K \geq 0$ &$K^2>0$&?\\
    \hline
  \end{tabular}
\end{table}
\medskip

We recall that the Kodaira dimension is invariant under \'etale covers
\citeMain[Theorem 10.9]{Iitaka}. Thus, for an \'etale tower of surfaces, 
the Kodaira dimension is constant and hence is that of the bottom surfaces.
So let us investigate surfaces of Kodaira dimension up to one.

\bigskip

Beginning with surfaces of Kodaira dimension $-1$, it is well--known
that rational surfaces are simply connected (see \citeMain[Corollaire
XI.1.2]{SGA1} for instance). Regarding ruled surfaces $X$ over a base
$C$, since for complete varieties, the fundamental group is a
birational invariant \citeMain[Example 5.2(h)]{MilEC} and $X$ is
birational to $C \times \P^1$, then, from \citeMain[Corollary
X.1.7]{SGA1}, $\pi_1(X) \simeq \pi_1(C) \times \pi_1(\P^1)$ and since
$\P^1$ is simply connected, we conclude that $\pi_1(X) = \pi_1
(C)$. More precisely, it can be checked, with a proof similar to that
of Theorem~\ref{pullback_relmin_model} above using Stein
factorisation, that any \'etale cover of a ruled surface is the
pullback of some \'etale cover of the base curve. It follows that the
splitting property of some infinite tower of a ruled surface would
comes from some of the base, so that we are reduced to the classical
problem on curves.

  \medskip
  
  Next, we continue with surfaces of Kodaira dimension $0$. First, K3
  surfaces are simply connected~\citeMain[Remark I.2.3]{Huybrechts}.
  Second, in odd characteristic, Enriques surfaces are known (see
  \citeMain{Dolgachev_Enriques}) to have an \'etale cover of degree 2 by a
  K3 surface. In characteristic $2$
  (see \citeMain{Dolgachev_Enriques} again) an Enriques surface is either
  covered by a surface which is birational to a K3 surface or by a non
  normal rational surface. Therefore, in any characteristic, an Enriques
  surface is covered by a simply connected surface and hence
  Proposition~\ref{prop:finiteness_of_pi_1} asserts that Enriques surfaces
  are bad candidates.
  Third, whereas abelian surfaces can have infinite towers
  of \'etale covers, such towers cannot satisfy a non trivial
  splitting property because each stage being an abelian surface (see
  \citeMain[Chapter 4 \S 18]{mumfordAV}), it cannot have more than
  $(\sqrt q + 1)^4$ rational points from Weil bound.  This means that
  the marked fundamental group of an abelian surface is finite.
  Fourth, hyperelliptic surfaces are a finite quotient of a product of
  two elliptic curves, hence are dominated under a finite map by an
  abelian surface whose marked fundamental group is finite. Hence,
  from Proposition~\ref{prop:finiteness_of_pi_1}, their marked
  fundamental group is finite too. Last, we are not able to conclude
  for quasi-elliptic surfaces.  \medskip

  Then, we end with surfaces of Kodaira dimension $1$, that is for
  elliptic surfaces.  Here, we are only able to prove, again with a
  proof very close to that of Theorem~\ref{pullback_relmin_model}
  above using Stein factorisation, that any \'etale cover of an
  elliptic surface is an elliptic surface whose base is an étale cover
  of the original base curve.

\subsection{Example: product of hyperelliptic curves} \label{sec:example}
The aim of this section is to provide an illustrative example of
application of Theorem~\ref{Critere_Phi} in which any computation
can be made explicit.  It turns out that, for most surfaces $X$ for
which the answer in the last column in Table~\ref{tab:classif} is not
negative, it is a hard task to compute the Galois invariants of
$\H^1 (\Xbar, \F_\ell)$ and $\H^2 (\Xbar, \F_\ell)$, and
even to bound from below the first one and to bound from above the
second one in an efficient way.

The authors acknowledge that they do succeed only in the case of the
product of two hyperelliptic curves of genus larger than $2$.
Unfortunately, for a product of two curves $C \times D$ with canonical
projections $p_C, p_D$ and a prime integer $\ell$, one can prove that
the existence of an infinite tower of $\ell$--étale covers splitting
totally at a set of closed points $\cP$ entails either the existence
of an infinite tower of $\ell$--étale covers of $C$ splitting totally
at $p_C(\cP)$ or the existence of an infinite tower of $D$ splitting
totally at $p_D(\cP).$

\begin{prop}\label{prop:CxD_is_a_bad_example}
  Let $C, D$ be two smooth curves. Consider the product
  $X = C \times D$ with canonical projections $p_C, p_D$. Let $P$ be a
  rational point of $C$ and $Q$ be a rational point of $D$ and
  $C_0 := p_D^* Q$ and $D_0 := p_C^* P$. Let $\ell$ be a prime integer
  and $Y$ be a connected smooth surface and $\pi : Y \rightarrow X$ be
  an $\ell$--étale Galois cover of $X$, then at least one of the two curves
  $\pi^* C_0$ or $\pi^* D_0$ is connected.
  In addition, if $\cP$ is a set of rational points of $X$ that
  is totally split in $Y$, then $p_C (\cP)$ (resp. $p_D(\cP)$)
  is totally split in $\pi^* C_0 \rightarrow C_0$ (resp. $\pi^* D_0
  \rightarrow D_0$).
\end{prop}

\begin{proof}
  \noindent {\bf Step 1.} Let us prove that the divisor $C_0 + D_0$
  is ample. Indeed, let $g$ be a positive integer larger than
  the genus of $C$ and that of $D$. Then we will prove that
  $(2g+1)(C_0 + D_0)$ is very ample.

  Clearly, $(2g+1)P$ is very ample on $C$ and hence provides a closed
  immersion $C \hookrightarrow \P^{n_C}$ for some positive integer
  $n_C$. Similarly we $(2g+1)Q$ provides a closed immersion
  $D \hookrightarrow \P^{n_D}$. This yields a closed immersion
  $C \times D \hookrightarrow \P^{n_C} \times \P^{n_D}$, then by Segré
  embedding, we get a closed immersion
  $\varphi : C \times D \hookrightarrow \P^N$ for some positive
  integer $N$ and $(2g+1)(C_0 + D_0) \sim \varphi^* \O_{\P^n}(1)$.

  \medskip

  \noindent {\bf Step 2.} Since the pullback of an ample divisor is
  ample (\citeMain[Prop 5.1.12]{ega2}), the divisor
  $\pi^* C_0 + \pi^* D_0$ is ample and hence has a connected support
  (\citeMain[Corollary III.7.9]{Hartshorne}).

  \medskip

  \noindent {\bf Step 3.} Suppose that both $\pi^* C_0$ and
  $\pi^* D_0$ are not connected. Then, using base change principle and
  since $\pi$ is a Galois cover, one deduces that both are $\ell$
  copies of their base. In addition, by the projection formula we have
  \[
    \pi^*C_0 \cdot \pi^* D_0 = \ell
  \]
  and hence each copy of $C_0$ in $\pi^* C_0$ meets a single copy of $D_0$
  in $\pi^* D_0$ with multiplicity $1$ and avoids the other ones.
  Thus, one deduces that $\pi^* C_0 + \pi^* D_0$ has a support which is
  isomorphic to the disjoint union of $\ell$ copies of the support of
  $C_0 + D_0$ on $X$. Therefore, it is non connected which contradicts
  its ampleness.

  \medskip

  \noindent {\bf Step 4.}
  Suppose that the pullback of $C_0$ is connected and
  that some rational point of $p_C (\cP)$ has a non rational
  inverse image $R$ under $\pi^* C_0 \rightarrow C_0$. Then, some elements
  of $\pi^* \cP$ lie in the fibre $\{R\} \times \pi^* D_0$ which does
  not contain any rational point. This contradicts the assumption
  that $\cP$ splits completely under $\pi$.
\end{proof}

Therefore, the application of our infiniteness-criterion in the sequel
yields an existence result that could have been proved with usual
class field theory in dimension one. Moreover, it is even true that,
under the usual assumptions, the marked (at a product of sets of
rational points) \'etale fundamental group of a product is the fibre
product of the corresponding marked \'etale fundamental groups, as
shown in \citeMain[Proposition 5.3]{schmidt3}. It follows that if
the marked \'etale fundamental group of a product
$\pi_{1}(C\times D,S\times T)$ is infinite, then at least one of the
two marked \'etale fundamental groups $\pi_{1}(C,S)$ and
$\pi_{1}(D,T)$ has to be infinite and this also holds for maximal
pro-$p$-quotients.

Thus, we emphasize that the
  example to follow is only illustrative. The objective in only to
  show that the objects involved in Theorem~\ref{Critere_Phi} can
  be explicitely computed.

\begin{prop} \label{Tour_hyperell}
Let $q$ be a power of an odd prime number. 
Let $f(t) \in \F_q[t]$ be the product of $2g_1+2$ distinct linear factors
and $g(t)\in \F_q[t]$ be the product of $g_2+1$ distinct irreducible quadratic factors.
Let $C : y^2=f(t)$ and $D : y^2=g(t)$ be the associated hyperelliptic curves.
Let $\rho$ be a positive integer such that :
\begin{itemize}
\item $2g_1+2+2\rho \leq \sharp C(\F_q)$;
\item $2\rho \leq \sharp D(\F_q)$;
\item  $2g_1+g_2\geq 3\rho + 2 + 2\sqrt{2g_1g_2+2+4\rho}$.
\end{itemize}
Then, there exists an infinite
\'etale tower of $C\times D$ in which some set of $4\rho$ rational points splits totally.
\end{prop}

Before proving this Proposition below, let us fix some notations and state two preliminary lemmas.
We begin by choosing the prime number $\ell = 2$, prime to $q$. 
We denote by $\sigma$ and $\tau$ the hyperelliptic involutions on $C$
and $D$ respectively. Those extend on $X=C\times D$ by
$\sigma(p, q) = (\sigma(p), q)$ and $\tau(p, q)=(p, \tau(q))$,
generating the abelian group
$A=\langle \sigma, \tau \rangle \simeq ({\mathbb Z}/2{\mathbb
  Z})^2$. We denote by $A.P$ the orbit of a point
$P=(p, q) \in C\times D$ under $A$. Then neither $p$ nor $q$ is a
Weierstrass point if, and only if, this orbit have order $4$.

\begin{lem} \label{lemma_r_T_hyperell} Let $P_1, \dots, P_{\rho}$ be
  $\rho$ rational points on $X=C\times D$ whose orbits under $A$ have $4$
  elements and are disjoints. Let $T$ be the union of these
  orbits. Then, $\sharp T= 4\rho$ and
\[r_T \leq 3\rho+1=\frac{3}{4}\sharp T +1.\]
\end{lem}
\begin{proof}
  We denote $P_i=(p_i, q_i) \in  C\times D$. Since all divisors of
  the form $p+\sigma(p)$ on $C$ are equivalent in $\CH_0(C)=\Jac_C$,
  and considering both horizontal curves $C\times\{q_1\}$ and
  $C\times \{\tau(q_1)\}$ on $X$, we have
\[(p_1, q_1)+(\sigma(p_1), q_1) \sim (p_i, q_1)+(\sigma(p_i), q_1)\]
and
\[(p_1, \tau(q_1))+(\sigma(p_1), \tau(q_1)) \sim (p_i, \tau(q_1))+(\sigma(p_i), \tau(q_1))\]
for any  $2\leq i \leq \rho$.

Summing these relations, we get, for any  $2\leq i \leq \rho$,
\begin{align}\label{r_T_rel_vert}
  (p_1, q_1)+(\sigma(p_1), q_1) + (p_1, \tau(q_1))+(\sigma(p_1), \tau(q_1))&\sim \\
  \nonumber
  (p_i, q_1)+(\sigma(p_i), q_1) +(p_i, \tau(q_1))&+(\sigma(p_i), \tau(q_1)).
\end{align}
 In the same way, considering both vertical curves $\{p_i\}\times D$ 
 and $\{\sigma(p_i)\} \times D$, we get
\begin{align}\label{r_T_rel_hor}
(p_i, q_1)+ (\sigma(p_i), q_1)+(p_i, \tau(q_1)) +(\sigma(p_i), \tau(q_1))
  &\sim \\
  \nonumber
  (p_i, q_i)+(\sigma(p_i), q_i) +(p_i, \tau(q_i))&+(\sigma(p_i), \tau(q_i))
\end{align}
for any  $2\leq i \leq \rho$. Both equations~(\ref{r_T_rel_vert}) and~(\ref{r_T_rel_hor}) give, for any $2\leq i \leq \rho$,
\begin{align*}
  (p_1, q_1)+(\sigma(p_1), q_1) + (p_1, \tau(q_1))+(\sigma(p_1), \tau(q_1))
  &\sim\\
  (p_i, q_i)+(\sigma(p_i), q_i) +(p_i, \tau(q_i))&+(\sigma(p_i), \tau(q_i))
  \end{align*}
  from which we deduce that, for any $2\leq i \leq \rho$, we can
  express one point in the orbit $A. P_i$ in term of the three other
  ones and those of $A. P_1$. It follows that
\[r_T \leq 4\rho-(\rho-1),\]
and the lemma is proved.
\end{proof}

Now, let us consider the Galois invariants of the groups
$\H^1({\Cbar\times \Dbar},\F_{2})$ and
$\H^2({\Cbar\times \Dbar},\F_{2})$, with coefficients in $\Lambda = \F_2$, for the action of
$G_{\F_q} = \textrm{Gal}(\Fqbar/ \Fq)$. The first one is
easily tackled. We have $\H^1 (\Cbar, \F_2) = \Jac_C[2]$, and
so as for $\Dbar$. From K\"{u}nneth formula
\[
  \H^1 ({\Cbar\times \Dbar}, \F_2) \simeq \H^1 (\Cbar, \F_2)
  \oplus \H^1 (\Dbar, \F_2)
\]
as a sum of $G_{\F_q}$-modules, we deduce
\begin{equation} \label{H1_CxD} \H^1 ({\Cbar\times \Dbar},
  \F_2)^{G_{\F_q}} \simeq \H^1 (\Cbar, \F_2)^{G_{\F_q}} \oplus
  \H^1 (\Dbar, \F_2)^{G_{\F_q}}\simeq \Jac_C[2]( \Fq) \oplus
  \Jac_D[2]( \Fq).
  \end{equation}
  The Galois invariants of
  $\H^2({\Cbar\times \Dbar}, \F_{2})$ also come from
  K\"unneth formula
  \[\H^2 ({\Cbar\times \Dbar}, \F_2)
  = \H^1 (\Cbar, \F_2) \otimes
\H^1 (\Dbar, \F_2) \oplus K
\]
as a sum of $G_{\F_q}$-modules, where $K$ is a $2$-dimensional vector space over $\F_{2}$ on which $G_{\F_q}$ acts trivially.
It follows that
\begin{equation}\label{dim_H2_CxD}
  \dim \H^2 ({\Cbar\times \Dbar}, \F_2)^{G_{\F_q}}
  = \dim (\H^1 (\Cbar, \F_2) \otimes
  \H^1 (\Dbar, \F_2) )^{G_{\F_q}} + 2.
\end{equation}
In the following Lemma (holding in fact for any prime $\ell$, prime to $q$), we denote by ${\rm Sp}_2(C)$ (resp.
${\rm Sp}_2(D)$) the spectrum in $\overline{\F}_{2}$ of the
Frobenius on $\Jac_C[2]\simeq\F_{2}^{2g_1}$ (resp. on
$\Jac_D[2]\simeq\F_{2}^{2g_2}$), by $m_{\lambda, C}$ the
dimension of the eigenspace of $\Jac_C[2]$ for the eigenvalue
$\lambda \in {\rm Sp}_2(C)$, and by
${\rm Sp}_2(C, D) = \{\lambda \in \overline{\F}_{2} ; \lambda
\in {\rm Sp}_2(C) \hbox{~and~} \lambda^{-1}\in {\rm Sp}_2(D)\}$.

\begin{lem}  \label{H2_CxD}
 We have
 \[
  \dim  {\left(\H^1 (\Cbar, \F_2) \otimes
\H^1 (\Dbar, \F_2)\right)}^{G_{\F_q}}
  = \sum_{\lambda \in {\rm Sp}_2(C, D)} m_{\lambda, C}.m_{\lambda^{-1}, D}.
  \]
\end{lem}

\begin{proof}
  We consider the decomposition of
  $\H^1 (\Cbar, \F_2) = \Jac_C[2]$ and that of
  $\H^1 (\Dbar, \F_2) = \Jac_D[2]$ as a sum of their
  \emph{characteristic} subspaces for the action of the Frobenius
\[\H^1 (\Cbar, \F_{2})=\bigoplus_{\lambda \in {\rm Sp}_2(C)} E_{\lambda}(C)\]
and
\[\H^1 (\Dbar, \F_2)=\bigoplus_{\mu \in {\rm Sp}_2(D)} E_{\mu}(D).\]
The decomposition of $\H^1 (\Cbar, \F_2) \otimes
\H^1 (\Dbar, \F_2)$ for the action of the Frobenius is then
\[\H^1 (\Cbar, \F_2) \otimes
\H^1 (\Dbar, \F_2) = \bigoplus_{(\lambda, \mu) \in {\rm Sp}_2(C)\times
  {\rm Sp}_2(D)}E_{\lambda}(C)\otimes E_{\mu}(D).\] The eigenvalue on
$E_{\lambda}(C)\otimes E_{\mu}(D)$ is the product $\lambda \mu$, and the
uniqueness of the decomposition of a vector in the decomposition above
shows that the invariants are those in
\[\bigoplus_{\lambda.\mu=1}E_{0, \lambda}(C)\otimes E_{0, \mu}(D),\]
where the $E_{0, \lambda}(C)$ and $E_{0, \mu}(D)$ are the
\emph{eigenspaces}, hence the Lemma.
\end{proof}

We can now prove Proposition~\ref{Tour_hyperell}.
\begin{proof}
  By hypothesis on $\sharp C(\F_q)$, there exists at least $2\rho$
  non-Weierstrass rational points on $C$. Since the hyperelliptic
  involution $\sigma$ on $C$ is defined over $\F_q$, one can choose
  $\rho$ distinct pairs of non $\sigma$-conjugated points
  $p_1, \sigma(p_1), \dots, p_{\rho}, \sigma(p_{\rho})$ in
  $C(\F_q)$. In the same way, since no Weierstrass point of $D$ is
  rational by assumption on $g$ and by hypothesis on $\sharp D(\F_q)$,
  one can choose $\rho$ distinct pairs of non $\tau$-conjugated points
  $q_1, \tau(q_1), \dots, q_{\rho}, \tau(q_{\rho})$ in $D(\F_q)$.  It
  follows that Lemma~\ref{lemma_r_T_hyperell} applies to
  $P_1=(p_1, q_1), \dots, P_{\rho}=(p_{\rho}, q_{\rho}) \in (C\times
  D)(\F_q)$, and we choose $T$ to be the union of the $A$-orbits of
  $P_1, \dots, P_{\rho}$.

By hypothesis on $f$, the $2$--torsion of $C$ is rational and hence
\[\Jac_C[2]( \Fq)\simeq \F_{2}^{2g_1}\]
on which the Frobenius acts by identity.
By hypothesis on $g$, we have 
\[\Jac_2[2](\F_{q^2}) \simeq \F_{2}^{2g_2},\]
on which the Frobenius acts by a block-diagonal matrix with
$\begin{pmatrix} 0&1\\1&0\end{pmatrix}$ on the diagonal.
Since $\begin{pmatrix} 0&1\\1&0\end{pmatrix}$ is conjugated to
$\begin{pmatrix} 1&1\\0&1\end{pmatrix}$ over $\F_2$, we deduce
from~(\ref{H1_CxD}) that
\begin{equation} \label{H1Gamma_hyperell} \H^1 ({\Cbar\times
    \Dbar}, \F_2)^{G_{\F_q}} \simeq \F_{2}^{2g_1}\oplus
  \F_{2}^{g_2},
\end{equation}
and from~(\ref{dim_H2_CxD}) together with Lemma~\ref{H2_CxD} that
 \begin{equation} \label{H2Gamma_hyperell}
 \dim \H^2 (\Xbar, \F_\ell)^{G_{\F_q}} = 2g_1 g_2 +2
 \end{equation}
 since the only eigenvalue on $C$ is $1$, with $m_{1, C}=2g_1$ and $m_{1, D}=g_1$.
 
 We can then deduce from Theorem~\ref{Critere_Phi},
 Lemma~\ref{lemma_r_T_hyperell}, (\ref{H1Gamma_hyperell}) and (\ref{H2Gamma_hyperell})
 that their exists an infinite $2$-tower in which $T$ splits as soon
 as
\[2g_1+g_2\geq 3\rho + 2 + 2\sqrt{2g_1g_2+2+4\rho},\]
completing the proof of Proposition~\ref{Tour_hyperell}.
\end{proof}

\section{Open problems}\label{sec:open_prob}
In the theory of codes from curves, the usual problem to get
asymptotically good families of codes reduces to construct sequences
of curves whose sequence of numbers of rational points grows
``quickly'', and whose sequence of genera grows ``slowly''.  To get
asymptotically good families of codes from surfaces, we first need
sequences of surfaces whose number of rational points goes to
infinity.  Next, the natural question is: {\em the number of rational
  points should be large, but compared to what?} What will play the
part of the genus in the case of surfaces?  In \citeMain{Papikian_Crelle},
Papikian studies families of surfaces whose number of rational
points is large compared to the sum of the Betti numbers.  Despite the
interest of this problem for itself, it is up to now unclear that such
a feature would permit to construct good codes.

In the present section, we show that for families of surfaces of
general type, both invariants $K^2$ and the topological (i.e. \'etale)
Euler characteristic could be an analog of the genus for this coding
theoretic problem.  In particular, since the sum of Betti number is
larger than the topological Euler characteristic, Papikian's
investigations are relevant for the coding theoretic setting.

\medskip

\noindent{\bf Caution.} The open problem introduced in what follows is
independent from the question of existence of infinite étale
towers and can be applied to any sequence of surfaces of general type
with very ample canonical class.

\subsection{From surfaces of general type to the domain of codes}

\subsubsection{The asymptotic domain of surfaces of general type}
We define here a domain ${\mathcal S}_q$ in
$({\mathbb R}\cup \{+\infty\})^2$ as follow.
\begin{defn} \label{Sq} The asymptotic domain of surfaces of general
  type ${\mathcal S}_q$ is the set of
  points $(\kappa, \chi) \in ({\mathbb R}\cup \{+\infty\})^2$ for
  which there exists a sequence $X_n$ of smooth projective absolutely
  irreducible surfaces defined over $\F_q$, such that:
\begin{itemize}
\item for any $n$, the canonical divisor $K_{X_n}$ is very ample;
\item $\sharp X_n(\F_q)$ goes to $+\infty$ as $n$ goes to $+\infty$;
\item the ratio $\frac{K_{X_n}^2}{\sharp X_n(\F_q)}$ goes to a limit
  $\kappa \in {\mathbb R_+\cup\{+\infty\}}$;
\item the ratio
  $\frac{\chi({\mathcal O}_{X_n})}{\sharp X_n(\F_q)}$ goes to a
  limit $\chi \in {\mathbb R}\cup\{+\infty\}$.
\end{itemize}
\end{defn}
\begin{rem}
Note that $K_{X_n}$ being ample, we have $\kappa \geq 0$, and that
$\chi\geq 0$ in case $q$ is odd, since from~\citeMain{YiGu} we have
$\chi({\mathcal O}_X)>0$ for any surface $X$ of general type in odd characteristic.
\end{rem}

\begin{rem}
Note also that by Noether formula:
\begin{equation}\label{eq:Noether1}
  12\chi({\mathcal O}_X) = K^2+\chi_{\textrm{\'et}},
\end{equation}
one can choose as parameters for this domain any pair among 
$\lim \frac{\chi({\mathcal O}_{X_n})}{\sharp X_n(\F_q)}$, $\lim \frac{K_{X_n}^2}{\sharp X_n(\F_q)}$ and $\lim \frac{\chi_{\textrm{\'et}}}{\sharp X_n(\F_q)}$.
\end{rem}

  To bound this domain, we use the following well--known inequality~(\ref{eq:Noether2}) 
listed for instance in~\citeMain{YiGu}:
\begin{equation}\label{eq:Noether2}
  5K^2+36 \geq \chi_{\textrm{\'et}}
\end{equation}
which, together with Noether formula~(\ref{eq:Noether1}), 
leads asymptotically to 
\[
  \chi \leq \frac{\kappa}{2}\cdot
\]
It follows that the domain ${\mathcal S}_q$ lies below the line
$\chi = \frac{\kappa}{2}$.  Moreover, Theorem~\ref{thm:universal} together
with Proposition~\ref{prop:minor_Gamma2} entail the following new
asymptotic bound.

\begin{thm}\label{thm:kappa_geq}
  For any point $(\kappa, \chi) \in {\mathcal S}_q$, we have
  $\kappa \geq \frac{1}{(q+1)^2}\cdot$
\end{thm}
\begin{proof}
  Let $X$ be a member of a family of surfaces whose parameters
  go to $(\kappa, \chi)$. From Theorem~\ref{thm:universal}, the linear
  system $\Gamma = (q+1)K_X$ is $(X(\F_q))$-interpolating. Hence
  by Proposition~\ref{prop:minor_Gamma2}, we have
  $\Gamma^2 \geq \sharp X(\F_q)$, that is
\[(q+1)^2K_X^2 \geq \sharp X(\F_q),\]
from which the Theorem follows.
\end{proof}

\subsubsection{Some maps between the domain ${\mathcal S}_q$ and the domain of codes}
We recall that the domain of codes ${\mathcal D}_q$ is the set of points $(\delta, R) \in [0, 1]^2$, such that there exists a family of 
$[N_n, k_n, d_n]_q$-codes for which:
\begin{itemize}
\item $N_n$ goes to infinity;
\item $\frac{k_n}{N_n}$ goes to $R$;
\item $\frac{d_n}{N_n}$ goes to $\delta$.
\end{itemize}
A well-known Plotkin bound asserts that this domain lies under the line from $(0, 1)$ to $(1-\frac{1}{q}, 0)$.

\begin{prop}
For any integer $2\leq g \leq q$, the affine map
\[\varphi_g : (\kappa, \chi) \longmapsto \left(\delta = 1-g(q+1)\kappa,
R=\frac{g(g-1)}{2} \kappa+\chi\right)\]
sends the part of ${\mathcal S}_q$ for which
$\kappa < \frac{1}{g(q+1)}$ into the domain of codes ${\mathcal D}_q$.
\end{prop}

  The map $\varphi_g$ is illustrated  in Figure~\ref{fig:phi2}.

\begin{rem}
We emphasize the following.
\begin{enumerate}
\item If the maps $\varphi_g$ are defined for any $g\geq 2$, they are irrelevant for $g\geq q+1$
since then, the part of ${\mathcal S}_q$ for which
$\kappa < \frac{1}{g(q+1)}$ is empty by Theorem~\ref{thm:kappa_geq}.
\item It is easily checked that the whole part $\{(\kappa, \chi) ;  \chi\leq \frac{\kappa}{2}\}$ does map
under $\varphi_g$ in the area below Singleton bound $R+\delta \leq 1$, and even under Plotkin bound adding the restriction $\kappa \geq \frac{1}{(q+1)^2}$ given by Theorem~\ref{thm:kappa_geq}.
\end{enumerate}
\end{rem}

\begin{proof} Let $g \geq 2$ and $X$ be a member of a family whose
  parameters go to $(\kappa, \chi)$ in ${\mathcal S}_q$.  We consider
  the code $C(X, {\mathcal P}, G)$ for ${\cP}=X(\F_q)$ and $G=gK$,
  whose length is $N=\sharp X(\F_q)$. Since $K$ is very ample,
  $\Gamma=|(q+1)K|$ is $\cP$-interpolating by
  Theorem~\ref{thm:universal}. Moreover, since $g\geq 2$, we have
  $G\cdot K=gK^2>K^2$ for the ample divisor $H=K$. Consequently, from
  Theorem~\ref{thm:dim}, this code has minimum distance at least
  \[
    d \geq N-\Gamma \cdot G=N-g(q+1)K^2,
  \]and hence its asymptotic relative
  distance satisfies
  \[
    \delta \geq 1-g(q+1)\kappa.
  \]
  In addition, its dimension is at least
  \[
    k \geq \frac{1}{2}G \cdot
    (G-K)+\chi(\O_X)=\frac{g(g-1)}{2}K^2+\chi(\O_X),
  \]
  hence, its
  asymptotic transmission rate satisfies
  \[
    R \geq \frac{g(g-1)}{2}\kappa+\chi,
  \]
    and the Proposition is proved.
  \end{proof}

  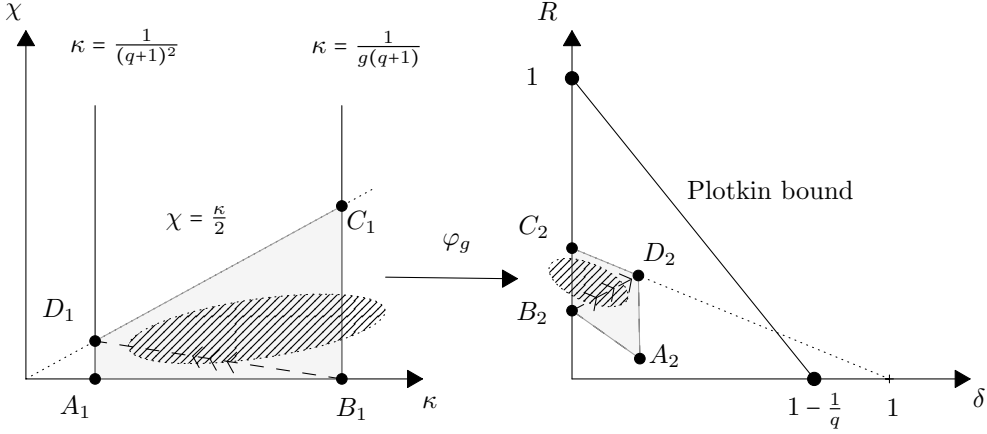
\begin{figure}
\definecolor{aqaqaq}{rgb}{0.6274509803921569,0.6274509803921569,0.6274509803921569}
\begin{tikzpicture}[line cap=round,line join=round,=triangle 45,x=1.0cm,y=1.0cm]
\clip(-3.7,-1.3) rectangle (10,5.2);
\fill[color=aqaqaq,fill=aqaqaq,fill opacity=0.10000000149011612] (-2.52,0.064) -- (-2.52,-0.44) -- (0.76,-0.44) -- (0.76,1.8608695652173912) -- cycle;
\fill[color=aqaqaq,fill=aqaqaq,fill opacity=0.10000000149011612] (3.82,0.468) -- (4.722,-0.17) -- (4.7025219811864085,0.936116529084277) -- (3.82,1.3) -- cycle;
\draw (-3.44,-0.44)-- (1.7,-0.44);
\draw (-3.44,-0.44)-- (-3.44,4.06);
\draw (3.82,-0.44)-- (3.82,4.06);
\draw (3.82,-0.44)-- (8.98,-0.44);
\draw (1.342,0.914)-- (2.962,0.894);
\draw [dotted] (3.82,1.3)-- (8.04,-0.44);
\draw (7.04,-0.44)-- (3.82,3.56);
\draw (-2.52,3.196)-- (-2.52,-0.44);
\draw (0.76,3.196)-- (0.76,-0.44);
\draw [dotted] (-3.44,-0.44)-- (1.16,2.08);
\draw (5.228,2.316) node[anchor=north west] {Plotkin bound};
\draw (-1.702,1.986) node[anchor=north west] {$\chi=\frac{\kappa}{2}$};
\draw (1.972,1.612) node[anchor=north west] {$\varphi_g$};
\draw [dash pattern=on 5pt off 5pt] (3.82,0.468)-- (4.722,-0.17);
\draw [dash pattern=on 5pt off 5pt] (4.722,-0.17)-- (4.7025219811864085,0.936116529084277);
\draw [color=aqaqaq] (-2.52,0.064)-- (-2.52,-0.44);
\draw [color=aqaqaq] (-2.52,-0.44)-- (0.76,-0.44);
\draw [color=aqaqaq] (0.76,-0.44)-- (0.76,1.8608695652173912);
\draw [color=aqaqaq] (0.76,1.8608695652173912)-- (-2.52,0.064);
\draw [color=aqaqaq] (3.82,0.468)-- (4.722,-0.17);
\draw [color=aqaqaq] (4.722,-0.17)-- (4.7025219811864085,0.936116529084277);
\draw [color=aqaqaq] (4.7025219811864085,0.936116529084277)-- (3.82,1.3);
\draw [color=aqaqaq] (3.82,1.3)-- (3.82,0.468);
\draw (-2.978,4.34) node[anchor=north west] {$\kappa=\frac{1}{(q+1)^2}$};
\draw (0.234,4.318) node[anchor=north west] {$\kappa=\frac{1}{g(q+1)}$};
\draw [dash pattern=on 5pt off 5pt] (0.76,-0.44)-- (-2.52,0.064);
\draw [dash pattern=on 5pt off 5pt] (-0.994160148113634,-0.17045831870449116) -- (-0.8574464097629168,-0.04122266671104367);
\draw [dash pattern=on 5pt off 5pt] (-0.994160148113634,-0.17045831870449116) -- (-0.9025535902370837,-0.33477733328895853);
\draw [dash pattern=on 5pt off 5pt] (-0.7658398518863665,-0.2055416812955098) -- (-0.6291261135356493,-0.07630602930206232);
\draw [dash pattern=on 5pt off 5pt] (-0.7658398518863665,-0.2055416812955098) -- (-0.6742332940098162,-0.36986069587997716);
\draw [dash pattern=on 5pt off 5pt] (-1.2224804443409014,-0.13537495611347128) -- (-1.085766705990184,-0.006139304120023785);
\draw [dash pattern=on 5pt off 5pt] (-1.2224804443409014,-0.13537495611347128) -- (-1.130873886464351,-0.2996939706979386);
\draw [dash pattern=on 5pt off 5pt] (3.82,0.468)-- (4.7025219811864085,0.936116529084277);
\draw [dash pattern=on 5pt off 5pt] (4.3778718376229575,0.7639122082725396) -- (4.338578420256205,0.5562947057549473);
\draw [dash pattern=on 5pt off 5pt] (4.3778718376229575,0.7639122082725396) -- (4.183943560930204,0.8478218233293308);
\draw [dash pattern=on 5pt off 5pt] (4.144650143563452,0.6402043208117385) -- (4.1053567261967,0.43258681829414625);
\draw [dash pattern=on 5pt off 5pt] (4.144650143563452,0.6402043208117385) -- (3.9507218668706985,0.7241139358685298);
\draw [dash pattern=on 5pt off 5pt] (4.611093531682464,0.8876200957333407) -- (4.571800114315711,0.6800025932157484);
\draw [dash pattern=on 5pt off 5pt] (4.611093531682464,0.8876200957333407) -- (4.41716525498971,0.9715297107901318);
\draw [rotate around={8.289404389025675:(-0.371,0.226)},dotted,fill=black,pattern=north east lines,pattern color=black] (-0.371,0.226) ellipse (1.724111057852998cm and 0.3937942861583842cm);
\draw [rotate around={-23.19859051364819:(4.04,0.842)},dotted,fill=black,pattern=north east lines,pattern color=black] (4.04,0.842) ellipse (0.5624530596492685cm and 0.25239937462051465cm);
\draw [fill=black,shift={(1.7,-0.44)},rotate=270] (0,0) ++(0 pt,3.75pt) -- ++(3.2475952641916446pt,-5.625pt)--++(-6.495190528383289pt,0 pt) -- ++(3.2475952641916446pt,5.625pt);
\draw[color=black] (1.928,-0.753) node {$\kappa$};
\draw [fill=black,shift={(8.98,-0.44)},rotate=270] (0,0) ++(0 pt,3.75pt) -- ++(3.2475952641916446pt,-5.625pt)--++(-6.495190528383289pt,0 pt) -- ++(3.2475952641916446pt,5.625pt);
\draw[color=black] (9.232,-0.687) node {$\delta$};
\draw [fill=black,shift={(-3.44,4.06)}] (0,0) ++(0 pt,3.75pt) -- ++(3.2475952641916446pt,-5.625pt)--++(-6.495190528383289pt,0 pt) -- ++(3.2475952641916446pt,5.625pt);
\draw[color=black] (-3.6,4.439) node {$\chi$};
\draw [fill=black,shift={(3.82,4.06)}] (0,0) ++(0 pt,3.75pt) -- ++(3.2475952641916446pt,-5.625pt)--++(-6.495190528383289pt,0 pt) -- ++(3.2475952641916446pt,5.625pt);
\draw[color=black] (3.49,4.461) node {$R$};
\draw [fill=black] (-2.52,-0.44) circle (2.0pt);
\draw[color=black] (-2.78,-0.819) node {$A_1$};
\draw [fill=black] (0.76,-0.44) circle (2.0pt);
\draw[color=black] (0.894,-0.841) node {$B_1$};
\draw [fill=black,shift={(2.962,0.894)},rotate=270] (0,0) ++(0 pt,3.75pt) -- ++(3.2475952641916446pt,-5.625pt)--++(-6.495190528383289pt,0 pt) -- ++(3.2475952641916446pt,5.625pt);
\draw [color=black] (8.04,-0.44)-- ++(-1.5pt,0 pt) -- ++(3.0pt,0 pt) ++(-1.5pt,-1.5pt) -- ++(0 pt,3.0pt);
\draw[color=black] (8.088,-0.819) node {$1$};
\draw [fill=black] (3.82,1.3) circle (2.0pt);
\draw[color=black] (3.314,1.579) node {$C_2$};
\draw [fill=black] (7.04,-0.44) circle (2.5pt);
\draw[color=black] (7.054,-0.863) node {$1-\frac{1}{q}$};
\draw [fill=black] (3.82,3.56) circle (2.5pt);
\draw[color=black] (3.292,3.581) node {$1$};
\draw [fill=black] (0.76,1.8608695652173912) circle (2.0pt);
\draw[color=black] (1.026,1.645) node {$C_1$};
\draw [fill=black] (3.82,0.468) circle (2.0pt);
\draw[color=black] (3.292,0.435) node {$B_2$};
\draw [fill=black] (4.722,-0.17) circle (2.0pt);
\draw[color=black] (5.052,-0.137) node {$A_2$};
\draw [fill=black] (-2.52,0.064) circle (2.0pt);
\draw[color=black] (-3.,0.479) node {$D_1$};
\draw [fill=black] (4.7025219811864085,0.936116529084277) circle (2.0pt);
\draw[color=black] (4.986,1.227) node {$D_2$};
\end{tikzpicture}
\caption{ The affine map
  $\varphi_g(\kappa, \chi)=(1-g(q+1)\kappa, \frac{g(g-1)}{2}\kappa
  +\chi)$ from ${\mathbb R}_+^2$ to ${\mathbb R}^2$.  The poorly known
  domain ${\mathcal S}_q$ is arbitrarily drawn as the hashed area in
  the left hand side, and its image under $\varphi_g$ as the hashed
  area in the right hand side. The part of ${\mathcal S}_q$ lying on
  the left of the line $\kappa=\frac{1}{g(q+1)}$ is contained inside
  the grey polygon $P_1=A_1B_1C_1D_1$. The point $A_1=(\frac{1}{(q+1)^2}, 0)$
  maps to $A_2=(1-\frac{g}{q+1}, \frac{g^2-g}{2(q+1)^2})$,
  $B_1=(\frac{1}{g(q+1)}, 0)$ maps to
  $B_2=(0, \frac{g^2-g}{2g(q+1)})$,
  $C_1=(\frac{1}{g(q+1)}, \frac{1}{2g(q+1)})$ maps to
  $C_2=(0, \frac{g^2-g+1}{2g(q+1)})$ and
  $D_1=(\frac{1}{(q+1)^2}, \frac{1}{2(q+1)^2})$ maps to
  $D_2=(1-\frac{g}{q+1}, \frac{g^2-g+1}{2(q+1)^2})$.  Hence the grey
  polygon $P_1$ maps onto the grey polygon $P_2=A_2B_2C_2D_2$ inside
  $[0, 1]^2$, below the Plotkin bound. The dotted line $(C_1D_1)$ of equation
  $\chi=\frac{\kappa}{2}$ on the left maps to the dotted line $(C_2D_2)$ of slope
  $-\frac{g^2-g+1}{2(q+1)^2}$ on the right; since $g \leq q$, this
  slope is less than $-\frac{1}{2}$, close to $-\frac{1}{2}$ if $g$ is
  close to $q$. The oriented dotted path $[B_1D_1]$ on the left is mapped on the
  oriented doted path $[B_2D_2]$ on the right.  }
 \label{fig:phi2}
\end{figure}

\subsubsection{A first new open problem}
Considering these maps, it follows that the study of ${\mathcal S}_q$ is of interest for coding theoretic purposes. 
More precisely, proving the existence of a family $(X_n)_{n \in {\mathbb N}}$ of surfaces of general
type with very ample canonical divisors with
\[\left(\kappa= \lim \frac{K_{X_n}^2}{\sharp X_n(\F_q)}, \chi=
  \lim \frac{\chi(\O_{X_n})}{\sharp X_n(\F_q)} \right)\] as close as
possible to the line $\chi =\frac{\kappa}{2}$ and with
$\frac{1}{(q+1)^2}\leq \kappa < \frac{1}{2(q+1)}$ would yield to a family
of codes near the line $(C_2D_2)$ of Figure~3. Moreover, the closer the point
$(\kappa, \chi)$ to
$D_1=(\frac{1}{(q+1)^2}, \frac{1}{2(q+1)^2})$, the closer the point to the Plotkin bound.

\bigskip

\begin{rem} \label{rem:Point_Prod} For any non-hyperelliptic curves
  $C_1$ and $C_2$ of genus $g_1\geq 3$ and $g_2\geq 3$, the surface
  $C_1\times C_2$ is of general type and the canonical divisor
  $K_{C_1\times C_2}=(2g_1-2)V+(2g_2-2)H$ is very ample using the
  Segre embedding. Moreover,
  $\chi({\mathcal O}_{C_1\times C_2})=(g_1-1)(g_2-1)$ and
  $K_{C_1\times C_2}^2=8(g_1-1)(g_2-1)$, so that this kind of surfaces
  with large genus and high number of rational points seems to provide
  good candidates since, asymptotically $\chi=\frac{\kappa}{8}$, whose
  slope is a quarter of the one of the line
  $\chi=\frac{\kappa}{2}$. Unfortunately, the asymptotic
  Drinfeld-Vladut bounds
  $\frac{\sharp C_i(\F_q)}{g_i} \leq \sqrt{q}-1$ for $i=1, 2$ yield to
  $\kappa \geq \frac{8}{(\sqrt{q}-1)^2}$, larger than
  $\frac{1}{2(q+1)}$ for any value of $q$. These product surfaces thus
  lie on the right of the line $\kappa =\frac{1}{2(q+1)}$, mapping
  under any $\varphi_g$ in the area $\delta <0$.
\end{rem}

\bigskip

\subsection{Asymptotic theory for surfaces~: algebraic and analytic side}
Kunyavskii--Tsfasman, Hindry--Pacheco, Zykin and Lebacque studied the
asymptotic behavior of $L$-functions of elliptic curves over rational
function field or of Zeta functions of varieties defined over a finite
field. Zykin put it in a general context but his study is well
adapted to the analytic context and his point of view does not suit
the algebraic properties of families.

In the case of curves, we compare the number of their rational points
with their genus. The beauty of the theory comes from the fact that
everything is closely connected : good families of global fields for
the algebraic point of view (unramified, with many points of small
norms) are good for the analytic point of view (giving rise to good
infinite Zeta functions) and they give rise to asymptotic good
families of codes and sphere packings.

For surfaces, we have (at least) two reasonable choices for what plays
the role of the genus : the sum of the $\ell$-adic Betti-numbers or
their alternate sum --their $\ell$-adic Euler-Poincar\'e
characteristic. The first appears when we consider the analytic side,
and the second appears naturally in the algebraic one, as it is
multiplied by the degree in \'etale covers. As far as we understand
now, the three points of view -- algebraic, analytic and
applications-- seem to involve different quantities and the relations
between them is not so clear as in the case of curves. An
asymptotically good family should be a family that is good for the
algebraic and analytic point of view, and one should be able to
construct from them good codes.

\textit{Questions :} What should be the definition of an
asymptotically good family? What should be the normalization? What
should be their algebraic and analytic properties?

\section*{Acknowledgements}
The authors express a deep gratitude to Olivier Dudas, David Madore,
Jade Nardi, Fabrice Orgogozo, Cédric Pépin, Michel Raynaud, Alexander
Schmidt and Jakob Stix, for inspiring discussions.

\bibliographystyleMain{amsplain}
\bibliographyMain{biblioMain}
\addresseshere

\end{document}